\documentclass[11pt]{article}%
\usepackage{amssymb,amsmath,amsfonts,amsthm,array,bm,bbm,color}%
\usepackage{graphicx}
\providecommand{\U}[1]{\protect\rule{.1in}{.1in}}

\numberwithin{equation}{section}

\newtheorem{theorem}{Theorem}[section]
\newtheorem{lemma}[theorem]{Lemma}
\newtheorem{corollary}[theorem]{Corollary}
\newtheorem{proposition}[theorem]{Proposition}
\newtheorem{remark}[theorem]{Remark}

\newtheorem{definition}[theorem]{Definition}
\newtheorem{hypothesis}[theorem]{Hypothesis}

\def\<{\langle}
\def\>{\rangle}
\def\d{{\rm d}}

\def\E{\mathbb{E}}

\def\N{\mathbb{N}}
\def\P{\mathbb{P}}
\def\R{\mathbb{R}}
\def\T{\mathbb{T}}
\def\Z{\mathbb{Z}}

\def\eps{\varepsilon}

\begin{document}

\title{Delayed blow-up by transport noise}
\author{Franco Flandoli\footnote{Email: franco.flandoli@sns.it. Scuola Normale Superiore of Pisa, Piazza dei Cavalieri 7, 56124 Pisa, Italy.}
\quad Lucio Galeati\footnote{Email: lucio.galeati@iam.uni-bonn.de. Institute of Applied Mathematics, University of Bonn, Germany.}
\quad Dejun Luo\footnote{Email: luodj@amss.ac.cn. Key Laboratory of RCSDS, Academy of Mathematics and Systems Science, Chinese Academy of Sciences, Beijing 100190, China, and School of Mathematical Sciences, University of the Chinese Academy of Sciences, Beijing 100049, China. } }

\maketitle

\vspace{-10pt}

\begin{abstract}
For some deterministic nonlinear PDEs on the torus whose solutions may blow up in finite time, we show that, under suitable conditions on the nonlinear term, the blow-up is delayed by multiplicative noise of transport type in a certain scaling limit. The main result is applied to the 3D Keller--Segel, 3D Fisher--KPP and 2D Kuramoto--Sivashinsky equations, yielding long-time existence for large initial data with high probability.
\end{abstract}



\section{Introduction}

In various applications, the time evolutions of certain quantities of interest are modelled by nonlinear partial differential equations (PDEs). Depending on the size of initial data, solutions to these nonlinear equations often exhibit a dichotomy of existence of global solutions or blow-up in finite time. More precisely, if the initial condition is below a certain threshold, then the equation admits a unique global solution; on the contrary, the solutions blow up in finite time for initial data above the threshold. It is commonly believed that this kind of blow-up may be suppressed by some background perturbation, either deterministic or stochastic; for examples on the former, we refer the reader to \cite{IXZ} and the references therein. In the present paper, we show that random perturbations of transport type work well in many cases and yield global existence for large initial data, with high probability (see Theorem \ref{Thm main} below for the precise meaning).

Consider a nonlinear deterministic equation on the torus $\mathbb{T}^{d} =\mathbb{R}^{d}/\mathbb{Z}^{d}$ of the form:
  \begin{equation}\label{PDE}
  \begin{cases}
  \partial_{t}u   = - (-\Delta)^{\alpha} u+F( u), \\
  u|_{t=0}  =u_{0},
  \end{cases}
  \end{equation}
where $u_{0}\in L^{2}\big(  \mathbb{T}^{d}\big)$ with $d\geq 2$, $\alpha\geq 1$ is fixed and $\Delta$ is the usual periodic Laplacian operator.
In \eqref{PDE}, $u$ can be vector valued functions, in which case $(-\Delta)^{\alpha}$ must be interpreted componentwise; $F$ is a nonlinearity satisfying suitable regularity assumptions, which will be specified below. We impose periodicity on initial conditions and solutions but we do not impose the zero mean condition, often asked in periodic equations to recover the easiest form of the Poincar\'{e} inequality; indeed, the zero mean condition is not preserved by $F$ in some of the examples treated below.
The intuition about the dynamics is that, if no blow-up occurs, dissipation tends to smooth solutions and $u( t,x)$ approaches its mean value $\int_{\T^d} u\left( t,x\right) \d x$ as $t$ increases; this is the typical behaviour of parabolic equations on compact manifolds.

In the recent works \cite{FlaLuo-20, Gal, FGL, LS, Luo20}, we considered some inviscid models (2D Euler or linear transport equations) perturbed by multiplicative noise of transport type. It turns out that these equations, originally of hyperbolic nature, converge weakly to parabolic equations under a suitable scaling of the noise; the larger the noise intensity, the higher the viscosity coefficient appearing in the limit equation. The same idea has recently been applied in \cite{FL19} to the vorticity formulation of 3D Navier-Stokes equations, showing that transport noise provides a blow-up control on the vorticity and gives long time existence, with large probability. This phenomenon has some similarity with the theory of stabilization by noise \cite{ArnoldCW, Arnold} in the finite dimensional setting, and is closely related to the mixing property and advection-induced dissipation enhancement which have been studied intensively in the literature, see \cite{Constantin, IXZ} and the references therein. A brief discussion of the relation between our results and those in \cite{Constantin, IXZ} will be given in Remark \ref{rem-discussion}.

Before stating our main result in Section \ref{subsec-main-result} on the effect of stochastic transport noise in preventing blow-up of solutions to abstract PDEs of the form \eqref{PDE}, we first mention several interesting models which will be considered in this paper.

\subsection{Some examples} \label{subsec-examples}

Our first example is the system of PDEs
\begin{equation}\label{keller-segel}\begin{cases}
\partial_t \rho = \Delta\rho  -\chi \nabla\cdot(\rho \nabla c)\\
-\Delta c = \rho -\rho_{\Omega}\\
\rho_{\Omega} =\int_{\Omega} \rho (x)\, \d x,
\end{cases}\end{equation}
which is commonly known as the Keller--Segel system; its exact description in terms of the general form \eqref{PDE} will be given in Section \ref{subsec keller-segel}. Although we will only deal with $\Omega=\mathbb{T}^d$, $d=2,3$ and periodic boundary conditions, for the sake of this preliminary discussion let us consider $\Omega$ to be a regular bounded domain of $\mathbb{R}^d$, with Neumann boundary condition.

System \eqref{keller-segel} is a simplified version, first considered in \cite{JL92}, of the model of chemotaxis introduced in \cite{Patlak, KS70, KS71}.
Here $\rho:\Omega\to\mathbb{R}$ describes the evolution of a bacterial population density whose motion is biased by the density of a chemoattractant $c:\Omega\to\mathbb{R}$ produced by the population itself; $\chi>0$ is a fixed sensitivity parameter.
The Keller--Segel system has received a lot of attention in mathematics literature as it exhibits the dichotomy behaviour described above: in $d=2$ the space $L^1(\Omega)$ is critical and it was shown in \cite{JL92} that for $\chi \rho_\Omega(0)$ below a critical threshold, global existence of regular solutions $(\rho,c)$ holds, but if $\Omega$ is a disk, there are examples of radially symmetric solutions blowing up in finite time; the blow-up mechanism is due to mass concentration and formation of Diracs for $\rho$.
The results from \cite{JL92} have been subsequently extended to the case $d=3$ in \cite{HMV97, HMV98}; for a detailed overview on the topic we refer the reader to \cite{Hor} and \cite[Chapter 5]{Perthame}.

The second example we will treat is the PDE
\begin{equation} \label{intro FisherKPP}
\partial_t u = \Delta u + u^2 -u
\end{equation}
which corresponds to \eqref{PDE} for the choice $\alpha=1$, $F(u)=u^2-u$; here $d=2, 3$. In the literature, equation \eqref{intro FisherKPP} is called the Fisher--KPP equation (see e.g. \cite{Fisher, KPP, McKean, Blath}), which has applications in spatial population genetics and ecology, modelling the spread of a beneficial allele subject to directional selection. In this case, spatially constant solutions $u(t,x)=y(t)$ satisfy the ordinary differential equation (ODE)%
\[
y^{\prime}(t) =y^{2}(t)  -y(t),
\]
so they exist globally in time and are bounded for initial data $y_0\leq1$, but they diverge to $+\infty$ in
finite time when $y_0>1$. For solutions with relatively small deviation from
its mean, we have a similar behaviour: global solutions for small
initial mean, blow-up for large initial mean. For not-nearly constant
solutions $u\left(  t,x\right)  $ the behavior can be more complicated, with a
balance between regions of high values that tend to explode in finite time and
regions of low values which help dissipating the higher ones.

Our last case of interest is the Kuramoto--Sivashinsky equation (cf. \cite{Kuramoto, Sivashinsky})
\begin{equation}\label{KSE}
  \partial_t u + \Delta^2 u + \Delta u + \frac12 |\nabla u|^2 =0.
\end{equation}
Differentiating the above equation and letting $\phi=\nabla u$ give us
\begin{equation}\label{KSE.1}
  \partial_t \phi + \Delta^2 \phi + \Delta \phi + \phi\cdot \nabla \phi =0 ,
\end{equation}
which is also called the  Kuramoto--Sivashinsky equation. One can reduce the equation \eqref{KSE} to \eqref{PDE} by taking $\alpha=2$ and $F(u)= -\Delta u -  \frac12 |\nabla u|^2$.  For $d=1$, global well posedness of \eqref{KSE.1} can be established exploiting the fact that the nonlinear term $\phi\partial_x \phi$ vanishes in energy type estimates; in this case, \eqref{KSE.1} has similar large scale behavior to that of the 1D viscous Burgers equation (or KPZ equation), see the discussion at the beginning of \cite{BCH}. However, due to the lack of a maximum principle, global well posedness of \eqref{KSE} in higher dimensions remains open, cf. \cite{BBT, LY20}; see also \cite{SellTab, Molinet} for some studies of \eqref{KSE} on thin 2D domains and the recent paper \cite{AmbMazz} for results on the time evolution of the radius of analyticity of solutions. Our results show that suitable random perturbations improve the well posedness of \eqref{KSE} in 2D, similar to those in \cite{FengMaz} in the deterministic setting.

\subsection{Our model, hypotheses and main result}\label{subsec-main-result}

Motivated by the above discussion, we consider the equation \eqref{PDE}  perturbed by a background  transport noise of the form
\begin{equation}\label{intro SPDE}
\partial_t u = -(-\Delta)^{\alpha} u + F(u) + \dot \eta(t) \circ \nabla u
\end{equation}
where $\eta(t,x)$ is some spatially divergence free noise and $\circ$ means that the stochastic differential will be understood in the Stratonovich sense.
Elementary considerations suggest that we can expect for \eqref{intro SPDE} similar results as for \eqref{PDE}: since Stratonovich noise obeys the classical chain rule and $\int_{\mathbb{T}^{d}} (  \eta(t)\cdot\nabla u ) u\, \d x=0$ due to the divergence free property of $\eta(t)$, equation \eqref{intro SPDE} enjoys the same a priori energy estimates as its deterministic counterpart; such estimates imply local solvability for arbitrary initial conditions and global solvability for small enough initial conditions.
But the techniques employed to establish blow-up for large initial data for \eqref{PDE} do not necessarily work for \eqref{intro SPDE}; hence, to be more precise, we expect the stochastic case to be not worse than the deterministic one.

The noise $\eta(t,x)$ adopted in this paper has the following form:
$$\eta(t,x) = \sqrt{C_d\nu} \sum_{k\in \Z^d_0} \sum_{i=1}^{d-1} \theta_k \sigma_{k,i}(x) W_{t}^{k,i} ,$$
where $C_d=d/(d-1)$ (see \eqref{key-identity} for a computation which justifies this choice) and $\nu>0$ is the noise intensity; $\theta = \{\theta_k \}_k \in \ell^2 = \ell^2(\Z_0^d)$, the space of square summable real sequences indexed by $\Z_0^d = \Z^d \setminus \{0\}$; up to relabelling $\nu$, we will always assume $\Vert \theta\Vert_{\ell^2}=1$. Moreover, it is enough to consider $\theta$ with finitely many nonzero components, and we shall always assume the symmetry property:
\begin{equation}\label{symmetry}
  \theta_k = \theta_l \quad \mbox{for all } k,l\in \Z^d_0 \mbox{ with } |k| =|l|.
\end{equation}
Next, the family $\{\sigma_{k,i}: k\in \Z^d_0, i=1,\ldots, d-1\}$ of periodic divergence free smooth vector fields are defined as follows. Let $\Z^d_0= \Z^d_+ \cup \Z^d_-$ be a partition of $\Z^d_0$ such that $\Z^d_+ = -\Z^d_-$. For any $k\in \Z^d_+$, take an ONB $\{a_{k,1}, \ldots, a_{k,d-1}\}$ of $k^\perp= \{y\in \R^d: y\cdot k=0\}$; let $a_{k,i}= a_{-k,i}$ for all $k\in \Z^d_-$. Then, we can define the vector fields
  $$\sigma_{k,i}(x) = a_{k,i} e^{2\pi {\rm i} k\cdot x}, \quad x\in \T^d,\, k\in \Z^d_0,\, i=1,\ldots, d-1, $$
where ${\rm i}$ is the imaginary unit. Finally, $\big\{W_{t}^{k,i}: k\in \Z^d_0,\, i=1,\ldots, d-1 \big\}$ are complex Brownian motions defined on a probability space $\left(  \Omega, \mathcal{F}, \mathbb{P} \right)  $, such that $\overline{W^{k,i}}=W^{-k,i}$ and $W^{k,i}$ and $W^{l,j}$ are independent whenever $k\neq - l$ or $i\neq j$. The above conditions can be summarised in the formula
\begin{equation}\label{covariation}
\big[ W^{k,i},W^{l,j} \big]_t= 2t\delta_{k,-l} \delta_{i,j}\quad \mbox{for all } k,l\in\Z^2_0,\, i,j\in \{1,\ldots, d-1\} .
\end{equation}
With these notations, the SPDE studied in this paper can be written more precisely as
\begin{equation}\label{PDE-perturbation} \begin{cases}
\d u =\big[  -(-\Delta)^{\alpha} u +F(u) \big]\, \d t+ \sqrt{C_d\nu}\sum_{k,i} \theta_k \sigma_{k,i}\cdot\nabla u\circ \d W_{t}^{k,i}, \\
u|_{t=0} = u_{0}.
\end{cases}\end{equation}
Here, $\sum_{k,i}$ stands for $\sum_{k\in \Z^d_0} \sum_{i=1}^{d-1}$.

\begin{remark}\label{rem-high-dim}
By \eqref{symmetry} and the definition of the vector fields $\{\sigma_{k,i} \}_{k,i}$, it is not difficult to show that the first equation in \eqref{PDE-perturbation} has the ``equivalent'' It\^o form (cf. \cite[Section 2.3]{Gal} or \cite[Section 2]{FL19} for the case $d=3$)
  $$\d u =\big[ -(-\Delta)^{\alpha} u+ \nu \Delta u + F(u) \big]\, \d t+ \sqrt{C_d\nu}\sum_{k,i} \theta_k \sigma_{k,i}\cdot\nabla u\, \d W_{t}^{k,i}. $$
We want to emphasize that, although the term $\nu \Delta u$ appears here, it does not mean that the dissipation has been enhanced at this stage. Take a sequence $\{\theta^N\}_{N\geq 1} \subset \ell^2$ such that
\begin{equation}\label{eq:cond-theta}
\|\theta^N \|_{\ell^2} =1 \quad (\forall\, N\geq 1), \quad \lim_{N\to\infty} \|\theta^N \|_{\ell^\infty} =0,
\end{equation}
and denote by $u^N$ the solution to the above equation corresponding to $\theta^N$. We shall show that, as $N\to \infty$, the martingale part will vanish in a suitable sense and we obtain a deterministic limit equation with enhanced dissipation:
  $$\partial_t u = - (-\Delta)^{\alpha} u+ \nu \Delta u + F (u). $$
This is the key for showing delayed blow-up of solutions to \eqref{PDE-perturbation}.
\end{remark}

We need some more notations to state the assumptions on the nonlinearity $F$. As usual, $L^2(\T^d)$ is the space of square integrable functions on $\T^d$ with the norm $\|\cdot \|_{L^2}$; we denote by $H^s(\T^d)\, (s\in \R)$ the usual (non-homogeneous) Sobolev space endowed with the norm $\|u \|_{H^s}= \|(1-\Delta)^{s/2} u \|_{L^2}$. The notation $\<\cdot, \cdot\>$ is used for the inner product in $L^2(\T^d)$ or the duality on $H^s(\T^d) \times H^{-s}(\T^d)$. In the sequel $a \lesssim b$ means that $a \leq C b$ for some unimportant constant $C>0$.

Now we are ready to state the assumptions on $F$, which are partly inspired by those from the variational method in SPDEs, see e.g. \cite[Section 4.1]{PR07} and \cite[Section 5]{LiuRoc2015}.   \medskip

\begin{hypothesis}
\begin{itemize}
\item[\rm(H1)] {\rm (Continuity)} There exist $ \beta_1 \geq 0$ and $\eta\in (0,\alpha)$ such that $F : H^{\alpha -\eta} \to H^{- \alpha}$ is continuous and
  \[ \|F(u)\|_{H^{-\alpha}} \lesssim  \big(1+ \|u\|_{L^2}^{\beta_1} \big) (1+\|u\|_{H^\alpha}) ; \]

\item[\rm(H2)] {\rm (Growth)} There exist $\beta_2\geq 0$ and $\gamma_2 \in (0,2)$ such that
  \[ | \langle F (u), u \rangle | \lesssim \big(1+ \|u\|_{L^2}^{\beta_2}\big) \big(1 + \| u \|_{H^\alpha}^{\gamma_2} \big) ; \]

\item[\rm(H3)] {\rm (Local monotonicity)} There exist $\beta_3, \kappa \geq 0$, $\gamma_3\in (0,2)$ such that $\beta_3 + \gamma_3\geq 2$, $\gamma_3 + \kappa \leq 2$ and
  \[ |\langle u-v, F(u)-F(v)\rangle | \lesssim \|u-v\|_{L^2}^{\beta_3} \|u-v\|_{H^\alpha}^{\gamma_3} \big( 1+\|u\|_{H^\alpha}^{\kappa} + \|v\|_{H^\alpha}^{\kappa} \big); \]

\item[\rm(H4)] There exists $\mathcal K\subset L^2(\mathbb{T}^d)$ convex, closed and bounded with the following property: for any $T>0$, we can find $\nu>0$ big enough such that the deterministic Cauchy problem
\begin{equation}\label{PDE enhanced viscosity}\begin{cases}
\partial_t u = - (-\Delta)^{\alpha} u+ \nu \Delta u + F (u), \\
u|_{t=0} = u_{0}
\end{cases}\end{equation}
admits a global solution $u\in L^2 (0,T; H^\alpha) \cap C([0,T];L^2)$ for any $u_0\in \mathcal K$, and moreover
\begin{equation}\label{unif bound enhanced viscosity}
\sup_{u_0\in \mathcal K} \sup_{t\in [0,T]} \Vert u(t;u_0,\nu)\Vert_{L^2}<\infty
\end{equation}
where $u(\,\cdot\,;u_0,\nu)$ denotes the unique solution to \eqref{PDE enhanced viscosity} with initial data $u_0$.
\end{itemize}
\end{hypothesis}

\begin{remark}\label{remark intro}
\begin{itemize}

\item[\rm(i)] In practice, conditions (H1)--(H3) are easy to check and they guarantee the existence and uniqueness of local solutions to both the deterministic equation \eqref{PDE} and the stochastic equation \eqref{PDE-perturbation}.

\item[\rm(ii)]
Although condition (H3) is sufficient for our purposes, let us point out that the proof works for the following more general condition:

\item[\rm{(H3$'$)}] There exist $N\in\mathbb{N}$ and nonnegative parameters $\beta_3^j,\gamma_3^j,\kappa_j,\kappa_j'$, $j=1,\ldots,N$ such that $\gamma_3^j\in (0,2)$, $\beta_3^j+\gamma_3^j\geq 2$, $\gamma_3^j + \kappa_j\leq 2$ for all $j$ and
\begin{equation*}
|\langle u-v, F(u)-F(v)\rangle|
\lesssim \sum_{j=1}^N \Vert u-v\Vert_{L^2}^{\beta_3^j} \Vert u-v\Vert_{H^\alpha}^{\gamma^j_3} \big(1+\Vert u\Vert_{H^\alpha}^{\kappa_j} +\Vert v\Vert_{H^\alpha}^{\kappa_j} \big) \big(1+\Vert u\Vert_{L^2}^{\kappa_j'} +\Vert v\Vert_{L^2}^{\kappa_j'} \big).
\end{equation*}
\item[\rm(iii)] Verification of hypothesis (H4) instead often requires nontrivial technical arguments, as will be shown in Section \ref{sec-example} for the PDEs given by \eqref{intro FisherKPP} and \eqref{KSE}. However, if the nonlinearity $F$ preserves the space of mean zero functions and we consider the dynamics restricted to this closed subspace of $L^2(\mathbb{T}^d)$, then it is rather immediate to verify (H4). Indeed, by (H2) and Young's inequality, any solution $u$ to \eqref{PDE enhanced viscosity} satisfies
\begin{align*}
  \frac{\d}{\d t} \| u \|_{L^2}^2 &
=  -2\Vert (-\Delta)^{\alpha/2} u\Vert_{L^2}^2 - 2 \nu \| \nabla u \|_{L^2}^2 + 2\langle F (u),u\rangle \\
  & \leq -2\Vert (-\Delta)^{\alpha/2} u\Vert_{L^2}^2 - 2 \nu \| \nabla u \|_{L^2}^2 + 2 C_1 \big(1 + \| u\|_{H^\alpha}^{\gamma_2} \big) \big(1+ \| u \|_{L^2}^{\beta_2} \big) \\
  & \leq -2\Vert (-\Delta)^{\alpha/2} u\Vert_{L^2}^2- 2\nu \| \nabla u \|_{L^2}^2 + 2^{2-\alpha} \Vert u\Vert_{H^\alpha}^2 + C_2\big(1+ \Vert u\Vert_{L^2}^{\tilde{\beta}} \big)
  \end{align*}
for some $C_1, C_2>0$ and $\tilde\beta= 2\beta_2/(2-\gamma_2) >0$. Using the fact $\| u\|_{H^\alpha}^2\leq 2^{\alpha-1} \big(\| u\|_{L^2}^2 + \|(-\Delta)^{\alpha/2} u\|_{L^2}^2 \big)$ and Poincar{\'e}'s inequality (with optimal constant $4\pi^2$), we obtain
  \begin{equation*}
  \frac{\d}{\d t} \| u \|_{L^2}^2
  \leq - 2 (4\pi^2 \nu - 1) \| u \|_{L^2}^2 + C_2\big(1+ \Vert u\Vert_{L^2}^{\tilde{\beta}} \big) = - \lambda_\nu \Vert u\Vert_{L^2}^2 + C_2\big(1+ \Vert u\Vert_{L^2}^{\tilde{\beta}} \big)
  \end{equation*}
where $\lambda_\nu := 2 (4\pi^2 \nu - 1)$. Observe that $\lambda_\nu$ can be as large as we want up to taking $\nu$ big enough. By the comparison principle it holds $ \| u(t) \|_{L^2}^2\leq x_t$, where $x_t$ is the solution with $x_0=\Vert u(0)\Vert_{L^2}^2$ to the ODE
  \[ \dot{x}_t = - \lambda_\nu\, x_t + C_2\big(1 + x_t^{\tilde{\beta}/2} \big).  \]
For any fixed $R\geq 0$, we can find $\lambda_\nu$ big enough, as well as a constant $C=C(\nu,R)$, such that the ODE starting from $x_0$ admits a global solution satisfying $x_t\leq C(\nu,R)$ for all $x_0\in [0,R]$. Hence hypothesis (H4) holds for $\mathcal K=\big\{ f\in L^2(\mathbb{T}^d): \int_{\mathbb{T}^d} f\,\d x =0, \Vert f\Vert_{L^2}\leq R \big\}$ for any $R\geq 0$.
\end{itemize}
\end{remark}

Our main result is that a sufficiently strong and rich noise improves the well posedness of the  equation \eqref{PDE}.
Given a deterministic $u_0\in L^2$, denote by $\tau=\tau(u_0,\nu,\theta)$ the random maximal time of existence of solutions $u(t;u_0,\nu, \theta)$ with trajectories in $C\left(  [0,\tau);L^{2}( \mathbb{T}^{d})  \right) $ to the stochastic problem \eqref{PDE-perturbation}.

\begin{theorem}\label{Thm main}
Assume $F$ satisfies (H1)--(H3) and $\mathcal K\subset L^2(\mathbb{T}^d)$ satisfies (H4). Then for any $T\in (0,\infty)$, for $\nu\in (0,\infty)$ as in (H4) depending on $T$, for every $\eps>0$, there exists $\theta\in \ell^2$ such that
\begin{equation}\label{thm main eq}
\mathbb{P} (  \tau(u_0,\nu,\theta) \geq T )  >1-\eps \qquad \forall \, u_0\in \mathcal K.
\end{equation}
Namely, with large probability uniformly over $u_0\in \mathcal K$, the maximal solution to \eqref{PDE-perturbation} with initial data $u_0$ has lifetime larger than  $T$.
\end{theorem}

We can deduce from Theorem \ref{Thm main} the following slightly stronger result; note that the assumptions (a) and (b) below are satisfied in many practical examples.

\begin{theorem}\label{thm-global-existence}
Assume the conditions (H1)--(H4). Moreover, assume
\begin{itemize}
\item[\rm(a)] the $L^2$-norm of the solution $u(t;u_0,\nu)$ to the deterministic equation \eqref{PDE enhanced viscosity} decreases exponentially fast, uniformly in $u_0\in \mathcal K$;
\item[\rm(b)] the stochastic equation \eqref{PDE-perturbation} admits a pathwise unique global solution for small initial condition.
\end{itemize}
Then, there exists $\theta\in \ell^2$ such that for  all $u_0\in \mathcal K$, the solution $u(t;u_0,\nu, \theta)$ to \eqref{PDE-perturbation} exists for all $t>0$ with large probability.
\end{theorem}

\begin{proof}
The idea of proof is the same as \cite[Theorem 1.6]{FL19} and thus we only provide a sketch here. First, by Theorem \ref{Thm main}, for $T$ big enough, there exists $\theta\in \ell^2$ such that for  all $u_0\in \mathcal K$, the solution $u(\cdot;u_0,\nu, \theta)$ to  the approximating stochastic equations \eqref{PDE-perturbation} has a life time greater than $T$, with large probability. The proof of Theorem \ref{Thm main} also shows that, with large probability, $u(\cdot;u_0,\nu, \theta)$ is very close to $u(\cdot;u_0,\nu)$ in the topology of $L^2(0,T; L^2)$. Combined with (a), we know that $\|u(t;u_0,\nu, \theta) \|_{L^2}$ is small enough for some $t= t(\omega)\in [T-1, T]$, with large probability. Now we  conclude the assertion  from condition (b).
\end{proof}

Although the main focus of this paper is the study of nonlinear equations of the form \eqref{intro SPDE}, our techniques provide interesting results also for $F\equiv 0$ and $\alpha=1$. In this case the equations in consideration become respectively
\begin{equation}\label{eq:STLE}
\partial_t u = \Delta u+ \dot \eta(t)\circ \nabla u
\end{equation}
and
\begin{equation}\label{eq:HE}
\partial_t u = (1+\nu) \Delta u
\end{equation}
subject to the initial condition $u\vert_{t=0} =u_0\in L^2$. As before, we denote their solutions respectively by $u(\cdot;u_0,\nu,\theta)$ and $u(\cdot;u_0,\nu)$; finally, let us set $\bar{u}_0=\int_{\T^d} u_0(x)\, \d x$.

\begin{corollary}\label{cor:relaxation-enhancing}
For every $\delta>0$, $\tau>0$ and $\eps>0$ there exist $(\nu,\theta)\in \R_+\times\ell^2$ such that
\begin{equation}\label{eq:relaxation-enhancing}
\P\big( \| u(\tau;u_0,\nu,\theta)-\bar{u}_0\|_{L^2} < \delta \big) >1-\eps \quad \text{uniformly over }\| u_0\|_{L^2}=1.
\end{equation}
\end{corollary}

\begin{proof}
By linearity we can assume $\bar{u}_0=0$; observe that $\P$-a.s. $t\mapsto \| u(t;u_0,\nu,\theta)\|_{L^2}$ is decreasing and $\| u(t;u_0,\nu)\|_{L^2} \leq \exp(-4\pi^2(1+\nu) t)$ (since $\| u_0\|_{L^2}=1$). Set $I=[\tau/2,\tau]$ and choose $\nu$ big enough such that
  $$\|u(\cdot;u_0,\nu) \|_{L^2(I;L^2)} \leq \sqrt{\frac\tau 2} \exp(-2\pi^2(1+\nu)\tau)\leq \frac\delta 2\sqrt{\frac\tau 2} .$$
Then we have
\begin{align*}
\P(\| u(\tau;u_0,\nu,\theta)\|_{L^2} \geq \delta)
& \leq \P\Big(\inf_{t\in I} \|u(t;u_0,\nu,\theta)\|_{L^2} \geq \delta\Big)\\
& \leq \P \bigg(\| u(\cdot;u_0,\nu,\theta)\|_{L^2(I;L^2)} \geq \delta\sqrt{\frac{\tau}{2}}\, \bigg) \\
& \leq \P \bigg(\| u(\cdot;u_0,\nu,\theta)-u(\cdot;u_0,\nu)\|_{L^2(I;L^2)} \geq \frac\delta 2\sqrt{\frac\tau 2} \, \bigg).
\end{align*}
Now take a sequence $\{\theta_N\}_N \subset \ell^2$ satisfying \eqref{eq:cond-theta}; by the same argument as in the proof of Theorem \ref{Thm main}, it holds
\[
\lim_{N\to\infty} \sup_{\| u_0\|_{L^2} \leq 1} \P\bigg(\| u(\cdot;u_0,\nu,\theta_N)-u(\cdot;u_0,\nu)\|_{L^2(0,T;L^2)} \geq \frac{\delta}{2}\sqrt{\frac{\tau}{2}}\,\bigg)=0
\]
which implies the conclusion.
\end{proof}

Let us mention that if $L^2$ is replaced by $H^{-s}$ for any $s>0$, then a similar statement can be proved for the inviscid version of \eqref{eq:STLE}, see also \cite[Section 6.3]{FGL}.

Before finishing this section, we give a short remark on the relation between our results and those in \cite{Constantin, IXZ}.

\begin{remark}\label{rem-discussion}
In \cite{Constantin}, for a suitable incompressible flow $b$ on a compact manifold $M$, the authors consider an advection diffusion equation of the form $\partial_t u = A\,b\cdot\nabla u + \Delta u$, which can be regarded as \eqref{eq:STLE} with $\dot\eta$ replaced by $A\, b$; here $A\in \R$. They study the relaxation enhancing property of $b$, which loosely speaking amounts to the solution $u$ getting arbitrarily close to its mean value $\bar{u}$ in arbitrarily short time, once $A$ is chosen big enough; in this sense, Corollary \ref{cor:relaxation-enhancing} gives an analogous statement with $A$ replaced by $(\nu,\theta)$. Let us however point out that in estimate \eqref{eq:relaxation-enhancing} the event of large probability depends on $u_0$, and we do not know whether it is possible to invert the quantifiers, namely to find an event of high probability on which the statement holds true for all $\|u_0\|_{L^2}=1$ at the same time. Theorem \ref{Thm main} instead is closer in spirit to the results from \cite{IXZ}, where the authors show that the addition of a deterministic transport term suppresses potential singularity in some nonlinear systems; here this is achieved, with large probability, by the multiplicative noise of transport type $\dot\eta$.
\end{remark}

Theorem \ref{Thm main} will be proved in Section \ref{sec3}, following the method of scaling limit from our recent papers  \cite{FlaLuo-20, Gal, FGL, LS}. We check in Section \ref{sec-example} that the nonlinearities $F$ given in the examples in Section \ref{subsec-examples} satisfy the hypotheses (H1)--(H4), the main efforts being devoted to the last one. In the appendix we show that, if a different scaling regime is considered, then in the limit only trivial, i.e. spatially constant, solutions to \eqref{PDE} can be obtained.

\section{Verifications of Examples} \label{sec-example}

In this section we check that conditions (H1)--(H4) are satisfied respectively for the Keller--Segel, Fisher--KPP and Kuramoto--Sivashinsky equations.

\subsection{The Keller--Segel system}\label{subsec keller-segel}

In order to check conditions (H1)--(H4) for system \eqref{keller-segel}, we need some preparations first. We set the parameter $\chi=1$ for simplicity, the other cases being similar. We will do the computations  only in the 3D case, as the 2D case is easier. For $f\in L^1(\T^3)$, we write $f_{\T^3}$ for the average $\int_{\T^3} f\,\d x$. Observe that if $\rho$ solves \eqref{keller-segel}, then it  has necessarily constant mean, i.e. $\rho_{\mathbb{T}^3} (t) = \rho_{\mathbb{T}^3} (0) =:\lambda$.

For any $f\in L^2(\mathbb{T}^3)$, there exists a unique $g\in H^2(\mathbb{T}^3)$ with zero mean such that $-\Delta g = f -f_{\mathbb{T}^3}$, which is usually denoted by $(-\Delta)^{-1} f$. For any $f\in L^2(\mathbb{T}^3)$, define the operator
\begin{equation*}
\nabla^{-1} f := \nabla (-\Delta)^{-1}(f-f_{\mathbb{T}^3});
\end{equation*}
it is possible to show that $\nabla^{-1}$ extends to a linear continuous operator from $H^s(\mathbb{T}^3)$ to $H^{1+s}(\mathbb{T}^3)$ for any $s\in \mathbb{R}$. Moreover by construction, for any $f\in L^2(\mathbb{T}^3)$, it holds $-\nabla\cdot\nabla^{-1} f = f-f_{\mathbb{T}^3}$.

With these notations in mind, setting $u(t):=\rho(t)-\rho_{\mathbb{T}^3}(t)= \rho(t)-\lambda$, it is easy to check that $\rho$ solves \eqref{keller-segel} if and only if $u$ is a zero mean function solving the equation
\[
\partial_t u = \Delta u - \nabla\cdot[(u+\lambda)\nabla^{-1}u]
\]
and using the property $-\nabla\cdot (\lambda\nabla^{-1}u)=\lambda u$ we can finally rewrite it as
\begin{equation}\label{keller-segel 2}
\partial_t u = \Delta u -\nabla\cdot[u \nabla^{-1} u] + \lambda u.
\end{equation}
From now on we will focus exclusively on the PDE \eqref{keller-segel 2}, which corresponds to \eqref{PDE} for the choice $\alpha=1$, $F(u)=-\nabla\cdot[u \nabla^{-1} u] + \lambda u$. This comes without loss of generality, as equation \eqref{keller-segel 2} is equivalent to the original system \eqref{keller-segel}, up to the knowledge of the parameter $\lambda$.

\begin{lemma}
The nonlinearity $F(u)=-\nabla\cdot[u\nabla^{-1}u] +\lambda u$ satisfies hypotheses (H1)--(H3) and any $u_0\in L^2(\mathbb{T}^3)$ with zero mean satisfies assumption (H4).
\end{lemma}

\begin{proof}
We can ignore the linear term $\lambda u$ and only focus on $G(u)=\nabla\cdot[u\nabla^{-1} u]$.\smallskip

\emph{Verification of} (H1). Clearly,
\[
\Vert G(u)-G(v)\Vert_{H^{-1}} = \Vert \nabla\cdot(u\nabla^{-1}u -v\nabla^{-1} v)\Vert_{H^{-1}} \leq \Vert u\nabla^{-1}u -v\nabla^{-1} v \Vert_{L^2};
\]
by H\"older's inequality and Sobolev embeddings $H^{7/4}(\T^3) \subset L^\infty(\T^3)$ and $H^{3/4}(\T^3) \subset L^4(\T^3)$, it holds
\begin{align*}
\Vert u\nabla^{-1}u -v\nabla^{-1} v\Vert_{L^2}
& \leq \Vert (u-v)\nabla^{-1} u\Vert_{L^2} + \Vert v\nabla^{-1}(u-v)\Vert_{L^2}\\
& \leq \Vert u-v\Vert_{L^2} \Vert \nabla^{-1} u\Vert_{L^\infty} + \Vert v\Vert_{L^4} \Vert \nabla^{-1}(u-v)\Vert_{L^4}\\
& \lesssim \Vert u-v\Vert_{L^2} \Vert \nabla^{-1} u\Vert_{H^{7/4}} + \Vert v\Vert_{H^{3/4}} \Vert \nabla^{-1}(u-v)\Vert_{H^{3/4}}\\
& \lesssim \Vert u-v\Vert_{L^2} (\Vert u\Vert_{H^{3/4}} + \Vert v\Vert_{H^{3/4}})
\end{align*}
which shows continuity of $G$ from $H^{3/4}$ to $H^{-1}$. Taking $v=0$ in the above estimates we obtain
\begin{align*}
\Vert G(u)\Vert_{H^{-1}} \lesssim \Vert u\Vert_{L^2} \Vert u\Vert_{H^{3/4}} \lesssim \Vert u\Vert_{L^2} \Vert u\Vert_{H^1} ,
\end{align*}
so that (H1) is satisfied with $\eta=1/4$, $\beta_1=1$.
\smallskip

\emph{Verification of} (H2). By the general formula
\begin{equation*}
\langle f,g\cdot\nabla f\rangle = - \frac{1}{2} \langle f^2, \nabla\cdot g\rangle,
\end{equation*}
which can be easily derived by integration by parts, it follows that
\begin{align*}
|\langle G(u),u\rangle |
=|\langle u,\nabla^{-1}u\cdot \nabla u\rangle|
= \frac{1}{2} \left\vert \int_{\mathbb{T}^3} u^3(x)\,\d x\right\vert
\lesssim \Vert u\Vert_{L^3}^3
\lesssim \Vert u\Vert_{H^{1/2}}^3
\lesssim \Vert u\Vert_{L^2}^{3/2} \Vert u\Vert_{H^1}^{3/2},
\end{align*}
where we used the Sobolev embedding $H^{1/2}(\mathbb{T}^3)\subset L^3(\mathbb{T}^3)$ and interpolation estimates. Therefore (H2) holds with $\beta_2= \gamma_2=3/2$.
\smallskip

\emph{Verification of} (H3). The estimates for (H1) also show that
\begin{align*}
|\langle G(u)-G(v),u-v\rangle|
& \leq \Vert G(u)-G(v)\Vert_{H^{-1}} \Vert u-v\Vert_{H^1}\\
& \lesssim \Vert u-v\Vert_{L^2} \Vert u-v\Vert_{H^1} (\Vert u\Vert_{H^{3/4}} + \Vert v\Vert_{H^{3/4}})\\
& \lesssim \Vert u-v\Vert_{L^2} \Vert u-v\Vert_{H^1} (\Vert u\Vert_{H^1} + \Vert v\Vert_{H^1}) ,
\end{align*}
so that (H3) holds with $\beta_3=1=\gamma_3=\kappa$.
\end{proof}
\begin{corollary}
For any $\lambda\in\mathbb{R}$ and any $R\geq 0$, hypothesis (H4) is satisfied for the choice
\[ \mathcal K_{R,\lambda} = \big\{ f\in L^2(\mathbb{T}^3): \Vert f-f_{\mathbb{T}^3}\Vert_{L^2}\leq R, f_{\mathbb{T}^3} = \lambda \big\}.
\]
\end{corollary}
\begin{proof}
If $\rho$ is a solution to \eqref{keller-segel} belonging to $\mathcal K_{R,\lambda}$, then $u=\rho-\rho_{\mathbb{T}^3}$ is a solution to \eqref{keller-segel} with $u_0$ being a mean zero function to \eqref{PDE} with $F(u)=-\nabla\cdot[u \nabla^{-1} u] + \lambda u$, $\Vert u_0\Vert_{L^2}\leq R$. The conclusion then follows from the fact that $F$ satisfies (H1) and point {\rm (iii)} of Remark \ref{remark intro}.
\end{proof}

\subsection{The Fisher--KPP equation}

Here $\alpha=1$, $F(u)=u^2-u$. As in the last section, we will consider only the 3D case. Observe that assumptions (H1)--(H3) are trivially satisfied by the linear term $-u$, therefore it is enough to verify them for the nonlinearity $G(u)=u^2$; this is the content of the next lemma.

\begin{lemma}
The hypotheses (H1)--(H3) hold for the nonlinearity $G(u)= u^2$.
\end{lemma}

\begin{proof}
\emph{Verification of} (H1). We first prove the continuity of $G$.
For any $\phi \in H^1(\T^3)$, by the H\"older inequality and the Sobolev embedding $H^{1/2}(\T^3) \subset L^{3}(\T^3)$, it holds
$$|\<u^2 - v^2, \phi\>|
\leq \|u+v\|_{L^3} \|u-v\|_{L^3} \|\phi \|_{L^3}
\lesssim \|u +v\|_{H^{1/2}} \|u-v\|_{H^{1/2}} \|\phi \|_{H^{1/2}}.$$
As a consequence we deduce
\begin{equation}\label{lem-FKPP-1}
\|G(u) - G(v)\|_{H^{-1}} \leq \|G(u) - G(v)\|_{H^{-1/2}}
\lesssim  (\|u \|_{H^{1/2}} + \|v \|_{H^{1/2}}) \|u-v\|_{H^{1/2}},
\end{equation}
which implies continuity of $G:H^{1/2}\to H^{-1}$. Taking $v=0$ in the above estimate and using interpolation inequalities  we also obtain
$$
\Vert G(u)\Vert_{H^{-1}}
\lesssim \Vert u\Vert_{H^{1/2}}^2
\lesssim \Vert u\Vert_{H^1} \Vert u\Vert_{L^2},
$$
so that (H1) is satisfied with $\eta=1/2,\, \beta_1=1$. \smallskip

\emph{Verification of} (H2). By the Sobolev embedding and interpolation estimates, we have
$$|\<G(u), u\>| = |\<u^2, u\>|
\leq \|u\|_{L^3}^3 \lesssim  \|u\|_{H^{1/2}}^3
\lesssim \|u\|_{L^2}^{3/2} \|u\|_{H^1}^{3/2}. $$
Therefore (H2) holds with $\beta_2=3/2$ and $\gamma_2=3/2$. \smallskip

\emph{Verification of} (H3). By the second inequality in \eqref{lem-FKPP-1} and the interpolation inequality,
\begin{align*}
|\<G(u)-G(v), u-v \>|
&\leq \|G(u)-G(v) \|_{H^{-1/2}} \|u-v \|_{H^{1/2}} \\
& \lesssim \|u-v \|_{H^{1/2}}^2 (\|u \|_{H^{1/2}} + \|v \|_{H^{1/2}}) \\
& \lesssim \Vert u-v\Vert_{H^1} \Vert u-v\Vert_{L^2} (\|u\|_{H^{1}}+\|v \|_{H^{1}}) ,
\end{align*}
which shows that (H3) holds for $\beta_3=\gamma_3= \kappa= 1$.
\end{proof}

The rest of this subsection is devoted to the proof that hypothesis (H4) holds for nonlinearity $F(u)= u^2- u$ with suitable $\mathcal K \subset L^2(\mathbb{T}^3)$. Recall the Poincar\'{e} inequality:
\[
\left\Vert v-\int_{\mathbb{T}^3}v(x)\,\d x \right\Vert _{L^{2}}^{2}
\leq (2\pi)^{-2} \left\Vert \nabla v\right\Vert _{L^{2}}^{2} \quad \forall\, v\in H^1(\mathbb{T}^3).
\]

\begin{proposition}\label{prop global large viscosity}
Fix $m_0 <1$ and $\sigma_0\in [0,+\infty)$. Then there exists $\nu=\nu(m_0,\sigma_0)$ big enough such that, for any initial data $u_0\in L^2(\mathbb{T}^3)$ satisfying
\[
\int_{\mathbb{T}^{3}}u_{0}(  x)\, \d x\leq m_0<1, \quad
\left\Vert u_{0}-\int_{\mathbb{T}^{3}}u_{0}(x)\, \d x\right\Vert _{L^{2}} \leq \sqrt{\sigma_0},
\]
the associated Cauchy problem
\begin{equation*} \begin{cases}
\partial_t u = (1+\nu)\Delta u +u^2 -u\\ u|_{t=0}=u_0
\end{cases} \end{equation*}
admits a global solution $u\in C([0,+\infty);L^2(\mathbb{T}^3))$; moreover, such solution satisfies
\begin{equation*}
\int_{\mathbb{T}^{3}}u(  t,x) \, \d x\leq1, \quad
\left\Vert u(  t,\cdot)  -\int_{\mathbb{T}^{3}}u(  t,x)\, \d x\right\Vert _{L^{2}}\leq \sqrt{\sigma_{0}},\quad
\left\Vert u(  t)  \right\Vert _{L^{2}}\leq 1+ \sqrt{\sigma_{0}}
\quad \forall\, t\geq 0.
\end{equation*}
\end{proposition}

\begin{proof}
\textbf{Step 1} (preliminary computations). For simplicity from now on we will write $\nu$ in place of $1+\nu$; denote the spatial average of $u$ by $u_{\mathbb{T}^3}$. Integrating $\partial_t u= \nu \Delta u + u^2 -u$ over $\mathbb{T}^3$ yields
\begin{align*}
\partial_{t}u_{\mathbb{T}^3}  &  =\nu \Delta u_{\mathbb{T}^3}%
+\|u\|_{L^2}^2 -u_{\mathbb{T}^3}
= \|u\|_{L^2}^2 -u_{\mathbb{T}%
^3}^{2}+u_{\mathbb{T}^3}^{2}-u_{\mathbb{T}^3}
= \left\Vert u-u_{\mathbb{T}^3}\right\Vert
_{L^{2}}^{2}+u_{\mathbb{T}^3}^{2}-u_{\mathbb{T}^3}.
\end{align*}
Hence, by the first equality,
\[
\partial_{t} (  u-u_{\mathbb{T}^3} )
=\nu\Delta ( u-u_{\mathbb{T}^3} )  + \left(  u^{2}-\|u\|_{L^2}^2 \right)  -\left(  u-u_{\mathbb{T}^3}\right)
\]
which implies
\begin{align*}
\frac{\d}{\d t}\left\Vert u-u_{\mathbb{T}^3}\right\Vert _{L^{2}}^{2}  &
= -2\nu\left\Vert \nabla\left(  u-u_{\mathbb{T}^3}\right)  \right\Vert
_{L^{2}}^{2}-2\left\Vert u-u_{\mathbb{T}^3}\right\Vert _{L^{2}}^{2}%
+2\int_{\mathbb{T}^3}\left(  u^{2}- \|u\|_{L^2}^2\right)
\left(  u-u_{\mathbb{T}^3}\right) \d x\\
&  \leq -2\nu\left\Vert \nabla u\right\Vert _{L^{2}}^{2}
+2\int_{\mathbb{T}^3}u^{2}\left(  u-u_{\mathbb{T}^3}\right)  \d x.
\end{align*}
We have
\begin{equation}\label{prop global large viscosity.1}
\aligned
\int_{\mathbb{T}^3}u^{2} (  u-u_{\mathbb{T}^3})\,\d x  &
=\int_{\mathbb{T}^3}(  u-u_{\mathbb{T}^3})^3 \,\d x+2\int_{\mathbb{T}^3}uu_{\mathbb{T}^3 } (u-u_{\mathbb{T}^3} ) \,\d x -u_{\mathbb{T}^3}^{2} \int_{\mathbb{T}^3} (u-u_{\mathbb{T}^3} ) \,\d x\\
& =\int_{\mathbb{T}^3} ( u-u_{\mathbb{T}^3} )^3 \,\d x +2u_{\mathbb{T}^3}\int_{\mathbb{T}^3}u (  u-u_{\mathbb{T}^3} ) \,\d x\\
& =\int_{\mathbb{T}^3} ( u-u_{\mathbb{T}^3} )^3\, \d x +2u_{\mathbb{T}^3}\int_{\mathbb{T}^3} (u-u_{\mathbb{T}^3} )^2\,\d x.
\endaligned
\end{equation}
Using the Sobolev embedding inequality and the interpolation inequality leads to
\begin{align*}
\bigg|\int_{\mathbb{T}^3} ( u-u_{\mathbb{T}^3} )^3\, \d x\bigg|
&\leq \|u-u_{\mathbb{T}^3} \|_{L^3}^3
\lesssim \|u-u_{\mathbb{T}^3} \|_{H^{1/2}}^3
\lesssim \|u-u_{\mathbb{T}^3} \|_{L^2}^{3/2} \|u-u_{\mathbb{T}^3} \|_{H^1}^{3/2} \\
&\leq C \|u-u_{\mathbb{T}^3} \|_{L^2}^6 + \|\nabla u \|_{L^2}^2.
\end{align*}
%
Combining this estimate with \eqref{prop global large viscosity.1} yields
\begin{align*}
\int_{\mathbb{T}^3}u^{2} (u-u_{\mathbb{T}^3}) \,\d x
& \leq \Vert \nabla u \Vert _{L^{2}}^{2} + C \Vert u-u_{\mathbb{T}^3} \Vert _{L^{2}}^{6}  +2u_{\mathbb{T}^3} \Vert u-u_{\mathbb{T}^3} \Vert _{L^{2}}^{2},
\end{align*}
and so overall applying Poincar\'e inequality we obtain
\begin{align*}
\frac{\d}{\d t}\left\Vert u-u_{\mathbb{T}^3}\right\Vert _{L^{2}}^{2}
& \leq -2\nu\left\Vert \nabla u\right\Vert _{L^{2}}^{2} +2\int_{\mathbb{T}^3}u^{2}\left(  u-u_{\mathbb{T}^3}\right)  \d x\\
& \leq -2(\nu-1) \Vert \nabla u\Vert_{L^{2}}^2 + C \Vert u-u_{\mathbb{T}^3} \Vert _{L^{2}}^{6}  +4 u_{\mathbb{T}^3} \Vert u-u_{\mathbb{T}^3} \Vert _{L^{2}}^{2}\\
& \leq (-\lambda_\nu + 4u_{\mathbb{T}^3}) \Vert u-u_{\mathbb{T}^3}\Vert_{L^2}^2 + C \Vert u-u_{\mathbb{T}^3} \Vert _{L^{2}}^{6}
\end{align*}
for the choice $\lambda_\nu = 8\pi^2 (\nu-1)$. Remark that the constant $C>0$ does not depend on $\nu$.

\textbf{Step 2} (global solutions). Summarizing the results of Step 1, we have
the system
\begin{align*}
\frac{\d}{\d t}u_{\mathbb{T}^3} &  =\left\Vert u-u_{\mathbb{T}^3}\right\Vert
_{L^{2}}^{2}+u_{\mathbb{T}^3}^{2}-u_{\mathbb{T}^3}, \\
\frac{\d}{\d t}\left\Vert u-u_{\mathbb{T}^3}\right\Vert _{L^{2}}^{2} &
\leq\left(  -\lambda_\nu + 4u_{\mathbb{T}^3}\right)  \left\Vert
u-u_{\mathbb{T}^3}\right\Vert _{L^{2}}^{2}+C\left\Vert u-u_{\mathbb{T}^3%
}\right\Vert _{L^{2}}^{6}.
\end{align*}
Setting $x(t)=u_{\mathbb{T}^3}(t)$, $y(t)  =\Vert u( t)  -u_{\mathbb{T}^3}(t) \Vert _{L^{2}}^{2}$, and writing $\lambda=\lambda_\nu$ for simplicity, we obtain a system of differential inequalities of the form
\begin{align*}\begin{cases}
x^{\prime} =y+x^{2}-x,\\
y^{\prime} \leq (  -\lambda+4x )  y+C y^3,\\
x( 0)   \leq m_{0}<1,\\
y( 0)  \leq \sigma_{0}
\end{cases}\end{align*}
with the additional information that $y(t)\geq 0$ for all $t\geq 0$. Our aim is to find $\lambda>0$ (equivalently $\nu>0$) big enough such that for any pair $(x(t),y(t))$ satisfying the above system it holds
\begin{equation*}
\sup_{t\geq 0} x(t) \leq 1,\quad \sup_{t\geq 0} y(t) \leq \sigma_0.
\end{equation*}
We can always take $m_0\in (0,1)$ and divide the proof in two cases.

\textbf{Case I.}
Assume $x(0)\in [0,m_0]$; observe that since $x'\geq x^2-x$, by comparison it must hold $x(t)\geq 0$ for all $t$. For any $\eps\in (0,1)$ such that $m_0+\varepsilon<1$, define
\begin{equation*}
T_\eps=\inf\big\{t\geq 0 : (x(t),y(t))\notin [0,m_0+\varepsilon)\times [0,\sigma_0+\eps)\big\}
\end{equation*}
with the convention $\inf\emptyset=+\infty$. Since $0\leq x(t)\leq 1$ for $t\in [0,T_\eps]$, $x^2(t)-x(t)\leq 0$; therefore
\begin{equation*}\begin{cases}
x'(t)\leq y(t), \\
y'(t) \leq \big[-\lambda + 4 + C(\sigma_0+1)^2 \big]\, y(t),
\end{cases}
\quad\, \forall\, t\in [0,T_\eps].\end{equation*}
Set $\gamma:= \lambda - 4 - C(\sigma_0+1)^2$ and assume $\lambda$ is big enough so that $\gamma>0$. Then, an application of Gronwall's inequality yields the estimates
\begin{align*}
y(t)\leq e^{-\gamma t} y(0)\leq \sigma_0,\quad
x(t)\leq m_0 + \int_0^t y(s)\, \d s
< m_0 + \frac{\sigma_0}{\gamma}\quad \forall\, t\in [0,T_\eps].
\end{align*}
The conclusion then follows if we can choose $\gamma$ such that $\sigma_0 \leq \eps \gamma$, as this implies $T_\eps=+\infty$. But this is equivalent to choosing small $\eps>0$ and large $\lambda>0$ such that
\begin{equation*}
\sigma_0\leq \eps \big[\lambda - 4 - C(\sigma_0+1)^2 \big],
\end{equation*}
which is always possible, for instance choosing $\eps= (1-m_0)/2$ and $\lambda=\lambda(m_0,\sigma_0)$ such that
\begin{equation*}
\lambda \geq 4+ C(\sigma_0+1)^2 + \frac{2\sigma_0}{1-m_0}.
\end{equation*}

\textbf{Case II.} Suppose now $x(0)<0$ and keep the same choice of $\eps, \lambda$ as above. Define
\begin{equation*}
\tau_\eps=\inf\big\{t\geq 0 : (x(t),y(t))\notin (-\infty,0)\times [0,\sigma_0+\eps)\big\},
\end{equation*}
then for $t\leq \tau^\eps$ it holds
\begin{equation*}
x'(t) \geq x^2(t) - x(t) \geq 0,\quad y'(t) \leq (-\lambda + 4x(t))y(t) + C y(t)^3\leq \big[-\lambda+C(\sigma_0+1)^2 \big] y(t).
\end{equation*}
By our choice of $\lambda$ it holds $-\lambda+C(\sigma_0+1)^2< 0$, thus
  $$y(t) < y(0) \leq \sigma_0\qquad \forall\,t\in [0,\tau^\eps). $$
As a consequence, $(x(t),y(t))\in [x(0),0]\times [0,\sigma_0]$ for all $t\in [0,\tau^\eps)$. Either $\tau^\eps=+\infty$, which implies the conclusion, or $\tau^\eps<+\infty$, in which case $(x(\tau^\eps),y(\tau^\eps))\in \{0\}\times [0,\sigma_0]$ and restarting the system at $(x(\tau^\eps),(y(\tau^\eps))$ we reduce to Case I.
\end{proof}

As a consequence we immediately deduce the following.
\begin{corollary}
For any $m_0<1$, $\sigma_0<\infty$, hypothesis (H4) is satisfied for the choice
\[
\mathcal K_{m_0,\sigma_0} = \big\{f\in L^2(\mathbb{T}^3): f_{\mathbb{T}^3}\leq m_0, \Vert f-f_{\mathbb{T}^3}\Vert_{L^2} \leq \sqrt{\sigma_0} \big\}.
\]
\end{corollary}
We conclude this section with the following trivial fact compared to the previous ones, but we state it as
a result, to collect all relevant facts in explicit statements.

\begin{proposition}
For any nonnegative initial condition $u_{0}\in L^{2}\left(  \mathbb{T}^{3}\right)  $ with
\[
\int_{\mathbb{T}^{3}}u_{0} ( x)\, \d x>1,
\]
independently of $\nu>0$, the solution to the deterministic equation
  $$\partial_t u = \nu\Delta u + u^2 -u$$
blows up in $L^2\left(  \mathbb{T}^{3}\right)$.
\end{proposition}

\begin{proof}
Similarly to the first step of the proof of Proposition \ref{prop global large viscosity}, integrating the equation on $\T^3$ yields
  $$\partial_t u_{\T^3}(t)= \int_{\T^3} u(t,x)^2 \,\d x- u_{\T^3}(t) \geq u_{\T^3}(t)^2 -u_{\T^3}(t),$$
where in the second step we have used Jensen's inequality. By the comparison principle, $ u_{\T^3}(t)\geq y(t)$, where $y$ is the solution to
  $$\left\{ \aligned
  y^{\prime}\left(  t\right)   &  = y^{2}\left(  t\right)  -y\left(  t\right),
   \\
  y\left(  0\right)   &  =y_{0}=\int_{\mathbb{T}^3} u_0(x\,)\d x>1,
  \endaligned \right. $$
which can be written explicitly as
  $$y(t)= \bigg[1- \Big(1- \frac1{y_0} \Big) e^t \bigg]^{-1}. $$
Note that the right-hand side explodes as $t\uparrow t_{0}= \log\frac{y_0}{y_0-1}$; therefore,
\[
\lim_{t\uparrow t_{0}} \|u(t) \|_{L^2} \geq \lim_{t\uparrow t_{0}} y(t) =+\infty.
\]
The proof is complete.
\end{proof}

\subsection{The 2D Kuramoto--Sivashinsky equation}

We turn to consider the Kuramoto-Sivashinsky equation \eqref{KSE}, which corresponds to $\alpha=2$ and the nonlinearity
  \begin{equation}\label{KSE-drift}
  F(u)= -\Delta u -\frac12 |\nabla u|^2.
  \end{equation}
Our purpose is to show that $F$ verifies the hypotheses (H1)--(H4); we start as usual with checking the first three.

\begin{lemma}\label{lem-KSE-1}
The nonlinearity in \eqref{KSE-drift} satisfies the hypotheses (H1)--(H3).
\end{lemma}

\begin{proof}
As before, we can ignore the linear part $-\Delta u$ which trivially satisfies the hypothesis and only focus on the nonlinearity $G(u)= | \nabla u|^2$.

\emph{Verification of }(H1). For any $u, v\in H^2(\T^2)$ and $\phi \in H^2(\T^2)$,
  $$\aligned
  |\<G(u)- G(v), \phi\>| &\leq |\<\nabla(u-v)\cdot\nabla u, \phi\>| + |\<\nabla v\cdot\nabla (u-v), \phi\>| \\
  &\leq \|\phi \|_{L^\infty} \|\nabla(u-v) \|_{L^2} \big(\|\nabla u\|_{L^2} + \|\nabla v\|_{L^2} \big) \\
  &\lesssim \|\phi \|_{H^{5/4}} \| u-v \|_{H^{1}} \big( \| u\|_{H^{1}} + \| v\|_{H^{1}} \big),
  \endaligned $$
where we used the Sobolev embedding $H^{5/4}(\T^2) \subset C(\T^2)$. This implies
  \begin{equation}\label{lem-KSE-1.2}
  \|G(u)- G(v) \|_{H^{-5/4}} \lesssim \|u-v \|_{H^{1}} \big(\|u\|_{H^{1}} + \|v\|_{H^{1}} \big),
  \end{equation}
which yields continuity of $G: H^{1}(\T^2)\to H^{-5/4}(\T^2)$. Next, taking $v=0$ in \eqref{lem-KSE-1.2} yields
\begin{equation*}
\Vert G(u)\Vert_{H^{-2}}
\leq \Vert G(u)\Vert_{H^{-5/4}}
\lesssim \Vert u\Vert_{H^{1}}^2 \lesssim \Vert u\Vert_{L^2} \Vert u\Vert_{H^2},
\end{equation*}
and we conclude that (H1) holds with $\eta = 1$ and $\beta_1=1$.

\emph{Verification of} (H2). By the above estimates,
\begin{equation*}
|\<G(u), u\>|
\leq \Vert G(u)\Vert_{H^{-{5/4}}} \Vert u\Vert_{H^{5/4}}
\lesssim \Vert u\Vert_{L^2} \Vert u\Vert_{H^2} \Vert u\Vert_{L^2}^{3/8} \Vert u\Vert_{H^2}^{5/8}
\lesssim \Vert u\Vert_{L^2}^{11/8} \Vert u\Vert_{H^2}^{13/8}.
\end{equation*}
Hence, (H2) holds with $\beta_2=11/8$, $\gamma_2= 13/8$.

\emph{Verification of} (H3). We have
  $$\aligned
  |\<G(u)- G(v), u-v\>| &\leq |\<\nabla(u-v)\cdot \nabla u, u-v\>| + |\<\nabla v\cdot \nabla(u-v), u-v\>| \\
  &\leq \|\nabla(u-v) \|_{L^2} \|u-v \|_{L^4} \big(\|\nabla u\|_{L^4} + \|\nabla v \|_{L^4} \big) \\
  &\lesssim \|u-v \|_{H^1} \|u-v \|_{H^{1/2}} \big(\|\nabla u\|_{H^{1/2}} + \|\nabla v\|_{H^{1/2}} \big)   \\
  &\lesssim \|u-v \|_{L^2}^{5/4} \|u-v \|_{H^2}^{3/4} (\|u\|_{H^2} + \|v\|_{H^2} ),
  \endaligned $$
and thus (H3) holds with $\beta_3= 5/4,\, \gamma_3=3/4$ and $ \kappa=1$.
\end{proof}

Now we turn to verify the hypothesis (H4) for which we need some preparations. First, we note that the limiting deterministic equation can be written as
  \begin{equation}\label{KSE-limit-eq}
  \partial_t u = -\Delta^2 u + (\nu -1) \Delta u -\frac12 |\nabla u|^2.
  \end{equation}
When $\nu=0$, the local existence of a unique mild solution to \eqref{KSE-limit-eq} has been established in \cite[Theorem 2.4]{FengMaz} (taking $\textbf{v} \equiv 0$ there). More precisely, for any $u(0)\in L^2(\T^2)$, there exist $T=T(\Vert u(0)\Vert_{L^2})>0$ and a unique $u\in C([0,T], L^2)$ which solves the mild formulation of \eqref{KSE-limit-eq} (with $\nu=0$):
  $$u(t)= e^{-t(\Delta^2 + \Delta)} u(0) -\frac12 \int_0^t e^{-(t-s)(\Delta^2 + \Delta)} \big(|\nabla u(s)|^2 \big)\,\d s; $$
moreover, by (2.20) in \cite{FengMaz},  $u$ satisfies
  \begin{equation}\label{KSE-gradient}
  \sup_{0<t<T} t^{1/4} \|\nabla u(t) \|_{L^2} \leq C(\|u(0) \|_{L^2}).
  \end{equation}
By classical a priori estimate, we know that $ u\in L^2(0,T; H^2)$.

We are interested in the case $\nu\geq 1$ in \eqref{KSE-limit-eq}, which is simpler than that with $\nu=0$ since the operator $\Delta^2 - (\nu -1) \Delta \geq \Delta^2$ is positive definite. Therefore, we have similar results as above and the constant in \eqref{KSE-gradient} can be shown to be independent of $\nu \geq 1$. Moreover, restarting the system \eqref{KSE-limit-eq} from any small time $t_0\in (0,T)$, as $u(t_0)\in H^{1}$, we conclude that the solution $u \in C([t_0, T], H^1) \cap L^2(t_0, T; H^3)$ (take a smaller $T$ if necessary). With these preparations, we will show that, for $\nu$ big enough, the solution $u$ to \eqref{KSE-limit-eq} indeed exists for all time.

\begin{lemma}\label{lem-1}
Let $u\in C([0,T]; L^2) \cap L^2(0,T; H^2)$ be the unique local solution to \eqref{KSE-limit-eq}. Denoting by $\bar{u}(t)= \int_{\T^2} u(x,t)\,\d x$, then $t\mapsto |\bar u(t)|$ is absolutely continuous on $[0,T]$ and for a.e. $t$,
  $$\frac{\d}{\d t}|\bar u(t)| \leq \| \nabla u(t) \|_{L^2}^2 .  $$
\end{lemma}

\begin{proof}
Integrating \eqref{KSE-limit-eq} on $\T^2$ yields
  $$\frac{\d}{\d t} \bar u(t) = - \frac12 \int_{\T^2} |\nabla u|^2 \,\d x= -\frac12 \|\nabla u \|_{L^2}^2.$$
Integrating on $[s,t] \subset (0,T)$ leads to
  $$\big| |\bar u(t)|- |\bar u(s)| \big| \leq |\bar u(t) - \bar u(s)| \leq \int_s^t \|\nabla u(r) \|_{L^2}^2 \,\d r.$$
This implies the absolute continuity of $t\to |\bar u(t)|$ and also the desired estimate.
\end{proof}

Next, we will estimate the $L^2$-norm of $\nabla u$: using equation \eqref{KSE-limit-eq},
  \begin{equation}\label{gradient-estimate}
  \aligned
  \frac{\d}{\d t} \|\nabla  u\|_{L^2}^2 &= -2 \<\Delta  u, \partial_t  u\>= -2\Big\<\Delta  u, (\nu-1)\Delta u - \Delta^2 u - \frac12 |\nabla  u|^2 \Big\> \\
  &= -2(\nu-1)\|\Delta u \|_{L^2}^2 - 2\|\nabla \Delta u \|_{L^2}^2 + \<\Delta u,  |\nabla  u|^2\>.
  \endaligned
  \end{equation}

\begin{lemma}\label{lem-2}
There exists a constant $C$ independent of $\nu$ such that
  $$\frac{\d}{\d t} \|\nabla  u\|_{L^2}^2 \leq -8\pi^2 (\nu-1)\|\nabla u \|_{L^2}^2 + C \|\nabla u \|_{L^2}^4 . $$
\end{lemma}

\begin{proof}
The key is to estimate the last term on the right-hand side of \eqref{gradient-estimate}. By H\"older's inequality we have
  $$|\<\Delta u,  |\nabla  u|^2\>| \leq \|\Delta u \|_{L^2} \|\nabla u \|_{L^4}^2 \lesssim \|\Delta u \|_{L^2} \|\nabla u \|_{H^{1/2}}^2 \lesssim \|\Delta u \|_{L^2} \|\nabla u \|_{L^2} \|\nabla u \|_{H^1}, $$
where in the second step we have used the Sobolev embedding $H^{1/2}(\T^2) \subset L^4(\T^2)$. Since $\|\Delta u \|_{L^2} \lesssim \|\nabla u \|_{H^1}$, we have
  $$|\<\Delta u,  |\nabla  u|^2\>| \lesssim \|\nabla u \|_{L^2} \|\nabla u \|_{H^1}^2 \lesssim \|\nabla u \|_{L^2}^2 \|\nabla u \|_{H^2}. $$
Noting that $\|\nabla u \|_{H^2} \lesssim \|\Delta \nabla u\|_{L^2}$, we obtain
  $$|\<\Delta u,  |\nabla  u|^2\>| \lesssim \|\nabla u \|_{L^2}^2 \|\nabla \Delta u\|_{L^2} \leq \|\nabla \Delta u\|_{L^2}^2 + C \|\nabla u \|_{L^2}^4. $$
Combining this estimate with \eqref{gradient-estimate} leads to
  $$\frac{\d}{\d t} \|\nabla  u\|_{L^2}^2 \leq -2(\nu-1)\|\Delta u \|_{L^2}^2 + C \|\nabla u \|_{L^2}^4 \leq -8\pi^2 (\nu-1)\|\nabla u \|_{L^2}^2 + C \|\nabla u \|_{L^2}^4, $$
where the last step follows from the Poincar\'e inequality.
\end{proof}

Denoting $x(t)= |\bar u(t)|$ and $y(t)= \|\nabla u(t) \|_{L^2}^2$, we deduce from Lemmas \ref{lem-1} and \ref{lem-2} a system of differential inequalities of the form
  \begin{equation}\label{differ-ineq-1}
  \left\{ \aligned
  x'&\leq y, \\
  y' &\leq -\lambda y + C y^2,
  \endaligned \right.
  \end{equation}
where
  $$\lambda=\lambda_\nu= 8\pi^2(\nu -1). $$
We are now ready to prove the following result.

\begin{proposition}\label{prop global large viscosity 2}
Fix $R\geq 0$. Then there exists $\nu=\nu(R)$ big enough such that, for any initial data $u_0\in L^2(\mathbb{T}^2)$ satisfying $\Vert u_0\Vert_{L^2} \leq R$, the associated Cauchy problem
\begin{equation*} \begin{cases}
\partial_t u = (\nu-1)\Delta u -\Delta^2 u - \frac12|\nabla u|^2, \\
u|_{t=0}=u_0
\end{cases} \end{equation*}
admits a global solution $u\in C\big([0,+\infty);L^2(\mathbb{T}^2) \big)$; moreover there exists $C=C(R)$ such that
\begin{equation*}
\sup_{t\geq 0} \Vert u(t)\Vert_{L^2}^2 \leq C.
\end{equation*}
\end{proposition}

\begin{proof}
The proof is similar to that of Proposition \ref{prop global large viscosity}. First, thanks to the estimate \eqref{KSE-gradient}, by running the solution for a small time $t_0$ (which only depends on $\Vert u_0\Vert_{L^2}$) and relabelling $R$, we can reduce ourselves to the case of $u_0\in H^1$ with $\Vert u_0\Vert_{H^1}\leq R$. In particular, setting $x(t)=|\bar{u}(t)|$, $y(t)=\Vert \nabla u(t)\Vert_{L^2}^2$, the pair $(x(t),y(t))$ satisfies the differential system \eqref{differ-ineq-1} with $x(0), y(0)\leq R$.

Define $T_\ast =\inf\big\{t\geq 0: (x(t),y(t))\notin [0,R+1)^2 \big\}$ with $\inf\emptyset=+\infty$, then by \eqref{differ-ineq-1} it holds
\begin{equation*}\begin{cases}
x'(t) \leq y(t), \\
y'(t) \leq -[\lambda - C(R+1)] y(t),
\end{cases}
\quad \forall \, t\in [0,T_\ast]. \end{equation*}
Setting $\gamma=\lambda-C(R+1)$, and choosing $\lambda$ sufficiently big so that $\gamma> 0$, by Gronwall's inequality it holds
\begin{equation*}
y(t)\leq e^{-\gamma t} y(0) \leq e^{-\gamma t} R,\quad
x(t)\leq x(0) + \int_0^t y(s)\,\d s \leq R + \frac{R}{\gamma} , \quad \forall\, t\in [0,T_\ast].
\end{equation*}
In order to conclude that $T_\ast=+\infty$, we need to choose $\gamma$ such that
\begin{equation*}
R+\frac{R}{\gamma}  < R+1
\quad \Longleftrightarrow \quad
\gamma > R,
\end{equation*}
which is further equivalent to
  $$\lambda > C(R+1) + R. $$
This is always possible up to taking $\lambda=\lambda_\nu(R)$ (respectively $\nu=\nu(R)$) big enough. For such choice, by Poincar\'e's inequality the associated solution $u$ satisfies
\begin{equation*}
\Vert u(t)\Vert_{L^2} \leq |\bar{u}(t)|+\Vert u(t)-\bar{u}(t)\Vert_{L^2} \leq (1+(2\pi)^{-1})(1+R)\quad \forall\, t\geq 0,
\end{equation*}
which gives the final claim.
\end{proof}

As a consequence we immediately deduce the following.
\begin{corollary}
Hypothesis (H4) is satisfied by $\mathcal K_R=\{f\in L^2(\mathbb{T}^2): \Vert f\Vert_{L^2} \leq R\}$ for any $R> 0$.
\end{corollary}

\section{Proof of the main result}\label{sec3}

The purpose of this section is to prove Theorem \ref{Thm main} under the hypotheses (H1)--(H4). Due to the nonlinearity $F$ in \eqref{PDE-perturbation}, the stochastic equation has only a local solution for general initial condition. Hence, we will make use of the cut-off technique. Let $g_{R}:[0,\infty)\mathbb{\rightarrow}\left[  0,1\right]  $ be a Lipschitz continuous cut-off function, equal to 1 on $\left[  0,R\right]  $ and equal to zero outside $\left[  0,R+1\right]  $. For a suitable parameter $\delta>0$ small enough, we write $g_{R}\left(  u\right)  $ for $g_{R}\left(  \left\Vert u\right\Vert_{H^{-\delta}}\right)  $; the choice of $\delta$ needed for our purposes will be discussed in Remark \ref{rem-hypo-H2} below, but at this stage of preliminary discussion any value $\delta>0$ is allowed. Consider the stochastic equation with cut-off:
\begin{align}\label{PDE-cut-off}\begin{cases}
  \d u  =\big[ -(-\Delta)^{\alpha} u+ g_R(u) F(u) \big]\, \d t+ \sqrt{C_d\nu}\sum_{k} \theta_k \sigma_{k,i}\cdot\nabla u\circ \d W_{t}^{k,i}, \\
  u|_{t=0} =u_{0}.
\end{cases}\end{align}
Recall that $\alpha\geq 1$. To simplify notations we introduce the operator $\Lambda= (-\Delta)^{1/2}$. Throughout this section $T>0$ is a fixed deterministic parameter and we will study well posedness of \eqref{PDE-cut-off} on the time interval $[0,T]$.

Now we want to transform the Stratonovich equation \eqref{PDE-cut-off} into the corresponding It\^o form. Under the symmetry assumption \eqref{symmetry} of $\theta$, this computation is by now standard, see \cite[Section 2.3]{Gal}; it relies on the elementary identity
  \begin{equation}\label{key-identity}
  \sum_{k,i} \theta_k^2\, \sigma_{k,i}(x)\otimes \overline{\sigma}_{k,i}(x)
  = \sum_{k,i} \theta_k^2\, a_{k,i}\otimes a_{k,i} = \frac{d-1}{d} \Vert \theta \Vert_{\ell^2}^2 I_d = \frac{d-1}{d} I_d,\quad \forall \, x\in\mathbb{T}^d,
  \end{equation}
where $I_d$ is the $d\times d$ unit matrix and in the last equality we used the assumption $\Vert \theta\Vert_{\ell^2}=1$. Thanks to \eqref{key-identity}, the first equation in \eqref{PDE-cut-off} can be rewritten in It\^o form as
  $$\d u =\left[ -\Lambda^{2\alpha}u + \nu \Delta u+ g_R(u) F(u) \right] \d t+ \sqrt{C_d \nu}\sum_{k,i} \theta_k \sigma_{k,i}\cdot\nabla u\, \d W_t^{k,i}.$$

\begin{definition}
Let $\big(\Omega, \mathcal F, (\mathcal F_t), \P\big)$ be a given stochastic basis on which a family $\{W^{k,i}\}_{k,i}$ of complex Brownian motions satisfying \eqref{covariation} is defined, $u_0\in L^2(\mathbb{T}^d)$. A process $u$ with trajectories in $C([0,T];L^2(\mathbb{T}^d)) \cap L^2(0,T; H^\alpha(\mathbb{T}^d))$ is a solution to the Cauchy problem \eqref{PDE-cut-off} if it is $\mathcal{F}_t$-adapted and for any $\phi\in H^\alpha (\mathbb{T}^d)$, with probability one it holds, for all $t\in [0,T]$,
\begin{align*}
  \langle u(t),\phi\rangle = &\, \langle u_0,\phi \rangle- \int_0^t \big[\langle \Lambda^{\alpha} u(s), \Lambda^{\alpha} \phi\rangle + \nu \langle \nabla u(s),\nabla\phi\rangle \big] \, \d s\\
  & + \int_0^t g_R(u(s)) \langle F(u(s)),\phi \rangle\, \d s- \sqrt{C_d \nu} \sum_{k,i} \theta_k \int_0^t \langle u(s),\sigma_{k,i}\cdot \nabla\phi\rangle \, \d W^{k,i}_s.
\end{align*}
\end{definition}

We remark that if $u \in L^2 \big(\Omega, C([0,T];L^2(\mathbb{T}^d))\cap L^2(0,T; H^\alpha(\mathbb{T}^d)) \big)$, then all the terms in the above equation are well defined. The next lemma clarifies the structure of solutions to \eqref{PDE-cut-off}. Let us recall that, for a given $L^2$-valued local martingale $M$, the quadratic variation process $[M]$ is the unique increasing process such that $\Vert M\Vert^2_{L^2} - [M]$ is a local martingale; see \cite{RozLot} for more details. In the following, for $p,q\in [1,\infty]$, we write $\|\cdot \|_{L^pL^q}$ for the norm in $L^p(0,T; L^q(\T^d))$; similarly, we use the notations $\|\cdot \|_{L^2H^s}$, $\|\cdot \|_{W^{1,2}H^s}$ and so on.

\begin{lemma}\label{sec3 lemma structure solutions}
Let $u$ be a solution to \eqref{PDE-cut-off}, then for any $\phi \in H^\alpha$ it holds
  \begin{equation}\label{sec3 decomposition solution}
  \langle u(t),\phi\rangle-\langle u_0,\phi\rangle = \langle v(t), \phi\rangle + \langle M(t),\phi\rangle,
  \end{equation}
where the process $v$ takes values in $W^{1,2}(0,T;H^{-\alpha})$ while $M(t)$ is an $L^2$-valued continuous local martingale; they are given respectively by
  \begin{equation}\label{sec3 martingale part}
  \aligned
  v(t) &= \int_0^t \big[- \Lambda^{2\alpha} u(s)+ \nu \Delta u(s) + g_R(u(s)) F(u(s)) \big]\, \d s, \\
  M(t) &= \sqrt{C_d\nu}\int_0^t \sum_{k,i} \theta_k\, \sigma_{k,i}\cdot \nabla u(s)\, \d W^{k,i}_s.
  \endaligned
  \end{equation}
Moreover there exists a constant $C=C(\nu,T)$ such that the process $v$ satisfies
  \begin{equation}\label{sec3 bound v}
  \Vert v\Vert_{W^{1,2}H^{-\alpha}} \leq C \big(1+ \Vert u\Vert_{L^\infty L^2}^{\beta_1} \big) (1+\Vert u\Vert_{L^2 H^\alpha}),
  \end{equation}
while the martingale part has quadratic variation
  \begin{equation}\label{sec3 martingale bracket}
  [M](t) = 2\nu\int_0^t \Vert \nabla u(s)\Vert_{L^2}^2\, \d s.
  \end{equation}
\end{lemma}

\begin{proof}
We can identify $\Lambda^{2\alpha} u$ with an element of $H^{-\alpha}(\T^d)$ by the relation $\langle \Lambda^{2\alpha} u,\phi\rangle = \big\langle \Lambda^{\alpha} u, \Lambda^{\alpha} \phi \big\rangle$; similarly, $\Delta u \in H^{\alpha-2} \subset H^{-\alpha}$. Therefore,
  $$ \int_0^T \| -\Lambda^{2\alpha} u(s)+\nu\Delta u(s) \|_{H^{-\alpha}}^2 \,\d s
  \lesssim (1+\nu^2) \int_0^T \| u(s)\|_{H^\alpha}^2 \d s = (1+\nu^2) \|u\|_{L^2 H^\alpha}^2. $$
Moreover, by hypothesis (H1) it holds
  \begin{align*}
  \int_0^T \Vert F(u(s)) \Vert_{H^{-\alpha}}^2\, \d s
  &\lesssim \int_0^T \big(1+\Vert u(s)\Vert_{H^\alpha}^2 \big) \big(1+ \Vert u(s)\Vert_{L^2}^{2\beta_1} \big)\, \d s \\
  &\lesssim \big(1+\Vert u \Vert_{L^2 H^\alpha}^2 \big) \big(1+ \Vert u \Vert_{L^\infty L^2}^{2\beta_1} \big).
  \end{align*}
Combining the above estimates shows that $v\in W^{1,2}(0,T;H^{-\alpha})$ and satisfies \eqref{sec3 bound v}.

Let us show that $M$ is a well defined $L^2$-valued stochastic integral; up to localisation, we can assume without loss of generality that $\mathbb{E}\big[\int_0^T \Vert u(s)\Vert_{H^\alpha}^2 \d s \big]<\infty$. By the assumption \eqref{covariation} on the noises $\{W^{k,i} \}_{k,i}$ and \eqref{key-identity}, it holds
  \begin{align*}
  \bigg[\sqrt{C_d\nu}\int_0^\cdot \sum_{k,i} \theta_k\, \sigma_{k,i}\cdot \nabla u(s)\, \d W^{k,i}_s \bigg](t)
  & = 2C_d\nu \int_0^t \sum_{k,i} \theta_k^2 \langle \sigma_{k,i}\cdot\nabla u(s), \overline{\sigma}_{k,i}\cdot \nabla u(s)\rangle\, \d s\\
  & = 2C_d\nu \int_0^t \int_{\mathbb{T}^d} \sum_{k,i} \theta_k^2\, \big\vert a_{k,i}\cdot \nabla u(s,x) \big\vert^2\, \d x\d s\\
  & = 2 \nu \int_0^t \int_{\mathbb{T}^d} |\nabla u(s,x)|^2\, \d x\d s
  \end{align*}
which gives \eqref{sec3 martingale bracket} and shows that $M$ is well defined, since
  \[
  \mathbb{E}\big( \Vert M(t)\Vert_{L^2}^2 \big) = \mathbb{E}\big( [M](t)\big)
  = 2\nu\, \mathbb{E} \bigg[\int_0^T \Vert \nabla u(s)\Vert_{L^2}^2\, \d s\bigg] <\infty.
  \]
Finally, by the divergence free property of $\sigma_k$ and integration by parts, it holds
  \begin{equation*}
  \langle M(t),\phi\rangle
  = \sqrt{C_d \nu} \int_0^t \sum_{k,i} \theta_k \langle \sigma_{k,i}\cdot \nabla u(s),\phi\rangle\,\d W^{k,i}_s
  = - \sqrt{C_d\nu} \int_0^t \sum_{k,i} \theta_k \langle u(s),\sigma_{k,i}\cdot \nabla \phi\rangle\,\d W^{k,i}_s
  \end{equation*}
which shows that any solution to \eqref{PDE-cut-off} decomposes as in \eqref{sec3 decomposition solution}.
\end{proof}

The following lemmas will be useful for several results.

\begin{lemma}\label{sec3 lemma compactness}
For any $\beta, \gamma, \varepsilon>0$ and any $p<\infty$, define the spaces
  \begin{equation*}
  \aligned
  \mathcal S &:=L^2(0,T;H^\alpha)\cap C([0,T];L^2)\cap C^\gamma([0,T];H^{-\beta}),\\
  \mathcal X &:= L^2(0,T;H^{\alpha-\eps})\cap L^p(0,T;L^2)\cap C([0,T];H^{-\varepsilon});
  \endaligned
  \end{equation*}
then we have the compact embedding $\mathcal S \hookrightarrow \mathcal X$. Moreover, for any $K \in [0,+\infty)$, the set
\begin{equation*}
\mathcal X_K := \Big\{ f\in \mathcal{X}: \sup_{t\in [0,T]} \Vert f(t)\Vert_{L^2}+\Vert f\Vert_{L^2 H^\alpha} \leq K \Big\}
\end{equation*}
is a closed subset of $\mathcal{X}$ and thus a Polish space with the metric inherited from $\mathcal{X}$.
\end{lemma}

\begin{proof}
Let $\{f_n \}_n$ be a bounded sequence in $\mathcal S$; then by Ascoli--Arzel\'a theorem, it admits a convergent subsequence (not relabelled for simplicity) $f_n\to f$ in $C([0,T];H^{-\beta})$; by the uniform bound in $C([0,T];L^2)$ and standard interpolation estimates, $f_n\to f$ in $C([0,T];H^{-\varepsilon})$ for any $\varepsilon>0$.  Similarly, by interpolation it holds
  \[
  \int_0^T \Vert f_n(t)-f(t)\Vert_{L^2}^4\, \d t\leq \sup_{t\in [0,T]} \Vert f_n(t)-f(t)\Vert^2_{H^{-\alpha}} \int_0^T \Vert f_n(t)-f(t)\Vert_{H^\alpha}^2 \d t\to 0,
  \]
which shows convergence in $L^4(0,T;L^2)$; convergence in $L^p(0,T;L^2)$ for general $p<\infty$ follows by interpolating this one with the uniform bound in $L^\infty(0,T;L^2)$. Convergence in $L^2(0,T;H^{\alpha-\eps})$ follows from convergence in $L^2(0,T;L^2)$ and the uniform bound in $L^2(0,T;H^\alpha)$.

Regarding the second claim, $\mathcal{X}_K$ being a closed subset of $\mathcal{X}$ follows immediately from the lower semicontinuity of $\Vert \cdot\Vert_{L^\infty L^2} +\Vert\cdot\Vert_{L^2 H^\alpha}$ in the topology of $\mathcal{X}$.
\end{proof}

\begin{lemma}\label{sec3 lemma continuity}
Under assumption (H1), the map $f\mapsto F(f)$ is continuous from $\mathcal{X}_K$ to $L^q(0,T;H^{-\alpha})$ for any $K<\infty$ and any $q\in [1,2)$.
\end{lemma}

\begin{proof}
First of all, by the definition of $\mathcal{X}_K$ and assumption (H1), $F(f)\in L^2 H^{-\alpha}$ for any $f\in\mathcal{X}_K$. Now fix $K<\infty$ and consider a sequence $f_n\to f$ in $\mathcal{X}_K$; in particular $f_n\to f$ in $L^2(0,T; H^{\alpha-\eps})$, which combined with the continuity of $F$ (taking $\eps<\eta$) implies that $F(f_n)$ converge in measure to $F(f)$ in $H^{-\alpha}$. We have the estimate
\begin{align*}
\Vert F(f_n)\Vert_{L^2 H^{-\alpha}}^2
= \int_0^T \Vert F(f_n(t))\Vert_{H^{-\alpha}}^2 \,\d t
\lesssim \int_0^T (1+\Vert f_n(t)\Vert_{L^2}^{2\beta_1}) (1+\Vert f_n(t)\Vert_{H^\alpha}^2) \,\d t
\lesssim C_{K,T},
\end{align*}
which shows that the sequence $\{F(f_n)\}_n$ is uniformly bounded in $L^2(0,T; H^{-\alpha})$. Convergence in $L^q(0,T; H^{-\alpha})$ for any $q<2$ then follows from an application of Vitali's theorem.
\end{proof}

Before showing the well posedness of the stochastic equation \eqref{PDE-cut-off} with cut-off, we make the following simple observation.

\begin{remark}\label{rem-hypo-H2}
By the interpolation inequality
  \[ \| u\|_{L^2} \leq \| u\|_{H^\alpha}^{\frac{\delta}{\alpha + \delta}}  \| u\|_{H^{-\delta}}^{\frac{\alpha}{\alpha +\delta}} ,  \]
condition (H2) also implies
  \[ \aligned | \langle F(u),u\rangle |
  &\lesssim \big(1+\| u\|_{H^\alpha}^{\gamma_2} \big) \Big(1+\| u\|_{H^\alpha}^{\beta_2 \frac{\delta}{\alpha+\delta}} \Big) \Big(1+\| u\|_{H^{-\delta}}^{\beta_2\frac{\alpha}{\alpha+\delta}} \Big) \\
  &\lesssim \Big(1 + \| u\|_{H^\alpha}^{\gamma_2 + \beta_2 \frac{\delta}{\alpha + \delta}} \Big) \Big(1+ \| u \|_{H^{-\delta}}^{\beta_2 \frac{\alpha}{\alpha + \delta}} \Big);
  \endaligned \]
therefore choosing $\delta$ sufficiently small, we can assume:
\begin{itemize}
\item[\rm(H2$'$)] There exist some $\tilde{\beta}_2>0$ and $\tilde{\gamma}_2 <2$ such that
  $$ | \langle F(u),u \rangle | \lesssim \big(1 + \| u\|_{H^\alpha}^{\tilde{\gamma}_2} \big) \big(1+ \| u\|_{H^{-\delta}}^{\tilde{\beta}_2} \big) .$$
\end{itemize}
With a slight abuse of notations we will still denote the parameters by $\beta_2$ and $\gamma_2$. From now on, when dealing with the equation \eqref{PDE-cut-off} with cut-off, we will always work with small $\delta>0$ such that {\rm (H2$'$)} holds.
\end{remark}

\begin{proposition}\label{prop-existence}
For any $\theta\in\ell^2$ and any $u_{0}\in L^{2}(  \mathbb{T}^{d})$, equation \eqref{PDE-cut-off} has a pathwise unique global strong solution $u$ with trajectories in $C\big([0,T]; L^2(\mathbb{T}^d) \big)\cap L^2\big(0,T; H^\alpha(\mathbb{T}^d) \big)$, satisfying
  \begin{equation}\label{sec3 energy bound}
  \P\mbox{-a.s.}, \quad \sup_{t\in [0,T]} \Vert u(t)\Vert_{L^2}^2 + \int_0^T \Vert \Lambda^{\alpha} u(t)\Vert_{L^2}^2\, \d t\leq C_1\big(1+ \Vert u_0\Vert_{L^2}^2 \big)
  \end{equation}
for some deterministic constant $C_1=C_1(T,\delta,R)$. Furthermore, for any $p>1$, $\beta>p+1$ and $\gamma<(p-1)/(2p)$, there exists another constant $C_2$, depending on all the above parameters but independent of $\|\theta\|_{\ell^2}$, such that the associated martingale part $M$ satisfies
  \begin{equation}\label{sec3 estimate martingale}
  \mathbb{E}\Bigg[\bigg(\sup_{0\leq s<t\leq T} \frac{\Vert M(t)-M(s)\Vert_{H^{-\beta}}}{|t-s|^\gamma}\bigg)^{2p}\Bigg]\leq C_2\, \Vert \theta\Vert_{\ell^\infty}^{2p}\, (1+\Vert u_0\Vert_{L^2}^2)^p.
\end{equation}
\end{proposition}

\begin{proof}
We divide the proof in several steps.

\textbf{Step 1} (Galerkin approximations and a priori estimates). For $N\in \N$, let $H_N$ be a finite dimensional subspace of $H= L^2(\T^d)$ spanned by $e^{2\pi {\rm i} k\cdot x},\, |k|\leq N$, and $\Pi_N: H\to H_N$ the orthogonal projection. Consider the finite dimensional SDE:
  \begin{equation*}
  \aligned
  \d u_N(t) &= \big[-\Lambda^{2\alpha} u_N(t)+ \nu \Delta u_N(t) + g_R(u_N(t))\Pi_N F(u_N(t)) \big]\,\d t \\
  &\quad\, - \sqrt{C_d\nu}\sum_{k,i} \theta_k \Pi_N (\sigma_{k,i} \cdot \nabla u_N(t))\, \d W^{k,i}_t
  \endaligned
  \end{equation*}
with initial condition $u_N(0)=\Pi_N u_0$. Note that the norms $\|\cdot \|_{L^2}$ and $\|\cdot \|_{H^s}$ in $H_N$ are equivalent; the Lipschitz regularity of $g_R: [0,\infty) \to [0,1]$ implies that $g_R(\|u_N\|_{H^{-\delta}})$ is also Lipschitz continuous in $u_N$. Using the hypotheses (H1) and (H3), one can show that the above SDE has continuous coefficients satisfying local monotonicity condition on $H_N$; therefore, local existence of solutions is granted (cf. \cite[Theorem 3.1.1]{PR07}). Moreover, It\^o's formula implies
  \begin{align*}
  \d \Vert u_N(t)\Vert_{L^2}^2
  & = 2\langle u_N(t),\d u_N(t)\rangle +\d [u_N](t)\\
  & = -2\|\Lambda^{\alpha} u_N(t)\|_{L^2}^2 \,\d t -2\nu \Vert \nabla  u_N(t)\Vert_{L^2}^2 \,\d t+ 2g_R(u_N(t)) \langle F(u_N(t)),u_N(t)\rangle\,\d t\\
  & \quad -2 \sqrt{C_d\nu}\sum_{k,i} \theta_k \<u_N(t), \Pi_N (\sigma_{k,i} \cdot \nabla u_N(t))\>\, \d W^{k,i}_t\\
  &\quad + 2C_d\nu\, \sum_{k,i} \theta_k^2 \Vert \Pi_N(\sigma_{k,i}\cdot \nabla u_N(t))\Vert_{L^2}^2 \,\d t.
  \end{align*}
Note that the martingale part vanishes because
  $$\<u_N(t), \Pi_N (\sigma_{k,i} \cdot \nabla u_N(t))\> = \<u_N(t), \sigma_{k,i} \cdot \nabla u_N(t)\> =0, $$
where we have used the divergence free property of $\sigma_{k,i}$. By \eqref{key-identity},  it holds
  \[
  2C_d \nu \sum_{k,i} \theta_k^2 \Vert \Pi_N(\sigma_{k,i} \cdot \nabla u_N(t))\Vert_{L^2}^2
  \leq 2C_d\nu \sum_{k,i} \theta_k^2 \Vert \sigma_{k,i} \cdot \nabla u_N(t)\Vert_{L^2}^2
  = 2\nu \Vert \nabla u_N(t)\Vert_{L^2}^2,
  \]
which implies that
  \begin{equation}\label{prop-existence.1}
  \frac{\d}{\d t} \Vert u_N(t)\Vert_{L^2}^2
  \leq -2 \Vert \Lambda^{\alpha} u_N(t)\Vert_{L^2}^2 + 2g_R(u_N(t)) \langle F(u_N(t)),u_N(t)\rangle.
  \end{equation}
Now, by hypothesis (H2$'$) in Remark \ref{rem-hypo-H2} and the definition of $g_R$, we have
  $$\aligned
  g_R(u_N(t)) |\langle F(u_N(t)), u_N(t)\rangle|
&\lesssim g_R(u_N(t)) \big(1+ \|u_N(t) \|_{H^\alpha}^{\gamma_2} \big) \big(1+ \|u_N(t) \|_{H^{-\delta}}^{\beta_2} \big)\\
& \lesssim \big(1+  R^{\beta_2} \big) \big(1+ \|u_N(t) \|_{H^\alpha}^{\gamma_2} \big) \\
  & \leq 1+ R^{\beta_2} +C\big( 1+ R^{\beta_2} \big)^{\frac{2}{2-\gamma_2}} + \frac1{2^\alpha} \|u_N(t) \|_{H^\alpha}^2 ,
  \endaligned $$
where in the last passages we used Young's inequality and the fact that $\gamma_2 <2$. Note that $\|u_N(t) \|_{H^\alpha}^2\leq 2^{\alpha-1} \big( \|u_N(t) \|_{L^2}^2 + \|\Lambda^\alpha u_N(t) \|_{L^2}^2 \big)$; substituting this estimate into \eqref{prop-existence.1} yields
  $$\frac{\d}{\d t} \Vert u_N(t)\Vert_{L^2}^2 \leq -\Vert \Lambda^{\alpha} u_N(t)\Vert_{L^2}^2 + C_{R,\beta_2,\gamma_2} + \|u_N(t) \|_{L^2}^2.$$
As a result, we can find a constant $C_1$ such that, $\mathbb{P}$-a.s.,
  \begin{equation}\label{prop-existence.2}
  \sup_{t\in [0,T]} \Vert u_N(t)\Vert_{L^2}^2 + \int_0^T \Vert \Lambda^{\alpha} u_N(t)\Vert_{L^2}^2\, \d t
  \leq C_1 \big(1+\Vert u_N(0)\Vert^2_{L^2} \big) \leq C_1 \big(1+\Vert u_0 \Vert^2_{L^2} \big).
  \end{equation}

We now pass to estimating the martingale term $M_N$ given by
  \[
  M_N(t) = \sqrt{C_d \nu} \int_0^t \sum_{k,i} \theta_k\, \Pi_N(\sigma_{k,i}\cdot \nabla u_N(s))\, \d W^{k,i}_s.
  \]
For any $\phi\in C^\infty(\mathbb{T}^2)$, $\langle M_N,\phi\rangle$ is a real valued martingale with quadratic variation satisfying
  \begin{align*}
  [\langle M_N,\phi\rangle](t)-[\langle M_N,\phi\rangle](s)
  & = C_d\nu\, \bigg[ \sum_{k,i} \theta_k \int_s^\cdot \langle u_N(r),\sigma_{k,i} \cdot \Pi_N\nabla\phi\rangle\, \d W^{k,i}_r \bigg](t)\\
  & = 2C_d \nu \sum_{k,i} \theta^2_k \int_s^t |\langle u_N(r),\sigma_{k,i} \cdot \Pi_N\nabla\phi\rangle|^2\, \d r,
  \end{align*}
where the second step follows from \eqref{covariation}. We have
  \begin{align*}
  [\langle M_N,\phi\rangle](t)-[\langle M_N,\phi\rangle](s) & \leq 2C_d\nu\Vert \theta\Vert_{\ell^\infty}^2 \int_s^t  \sum_{k,i} |\langle  \sigma_{k,i}, u_N(r) \Pi_N \nabla\phi\rangle|^2\, \d r\\
  & \leq 2C_d \nu\Vert \theta\Vert_{\ell^\infty}^2 \int_s^t \Vert u_N(r)\Pi_N\nabla \phi\Vert_{L^2}^2\, \d r\\
  & \leq 2C_d \nu C_1 \Vert \theta\Vert_{\ell^\infty}^2 \Vert \Pi_N\nabla \phi\Vert_{L^\infty}^2 \big(1+\Vert u_0\Vert_{L^2}^2 \big) |t-s| ,
  \end{align*}
where in the second line we used the fact that $\{\sigma_{k,i} \}_{k,i}$ is an (incomplete) orthonormal system in $L^2(\mathbb{T}^d; \mathbb{R}^d)$. As a consequence, for any $p>1$ and $\beta>1+p$, we can use the Jensen and Burkholder-Davis-Gundy inequalities to estimate $\mathbb{E} \big[\Vert M_N(t)-M_N(s)\Vert_{H^{-\beta}}^{2p} \big]$ as follows:
  \begin{align*}
  \mathbb{E}\big[\Vert M_N(t)-M_N(s)\Vert_{H^{-\beta}}^{2p} \big]
  & \lesssim \sum_k (1+|k|^2)^{-\beta} \mathbb{E} \big[|\langle M_N(t)-M_N(s),e_k\rangle|^{2p} \big]\\
  & \lesssim \Vert \theta\Vert_{\ell^\infty}^{2p} \big(1+\Vert u_0\Vert_{L^2}^2 \big)^p |t-s|^p \sum_k (1+|k|^2)^{-\beta} |k|^{2p}\\
  & \lesssim \Vert \theta\Vert_{\ell^\infty}^{2p} \,\big(1+\Vert u_0\Vert_{L^2}^2 \big)^p\, |t-s|^p,
  \end{align*}
where $e_k(x)= e^{2\pi {\rm i} k\cdot x}$ and we used the fact $\Vert \Pi_N \nabla e_k \Vert_{L^\infty} \lesssim |k|$; the sum over $k$ is convergent since $\beta>p+1$. Estimate \eqref{sec3 estimate martingale} with $M$ replaced by $M_N$ then readily follows from an application of Kolmogorov's continuity modification theorem.

\textbf{Step 2} (weak existence). Note that $u_N=u_N(0)+ v_N+M_N$, where
  $$v_N(t) = \int_0^t \big[-\Lambda^{2\alpha} u_N(s)+ \nu \Delta u_N(s) + g_R(u_N(s)) F(u_N(s)) \big]\,\d s $$
satisfies, similarly to \eqref{sec3 bound v},
  \begin{equation*}
  \Vert v_N\Vert_{C^{1/2 } H^{-\alpha}} \leq \Vert v_N\Vert_{W^{1,2} H^{-\alpha}} \lesssim \big(1+\Vert u_N\Vert_{L^\infty L^2}^{\beta_1} \big) (1+\Vert u_N\Vert_{L^2 H^\alpha}).
  \end{equation*}
Combining this fact with the above estimates \eqref{prop-existence.2}, we deduce that there exist $p>1$, $\beta,\gamma>0$ such that
  \begin{equation*}
  \sup_{N\geq 1} \mathbb{E}\big[ ( \Vert u_N\Vert_{L^2 H^\alpha} + \Vert u_N\Vert_{L^\infty L^2} + \Vert u_N\Vert_{C^\gamma H^{-\beta}} )^p \big] <\infty.
  \end{equation*}
Let $\mu_N$ be the law of $u_N,\, N\geq 1$; then by Lemma \ref{sec3 lemma compactness} and Prokhorov's theorem (see \cite[p.59, Theorem 5.1]{Billingsley}), the family $\{\mu_N \}_N$ is tight in $\mathcal X$. By estimate \eqref{prop-existence.2}, we can find $K$ large enough such that the laws $\{ \mu_N \}_N$ are all supported on $\mathcal{X}_K$, thus tight therein as well.

The existence of a weak solution then follows from by now classical arguments, pioneered in this setting in \cite{FlaGat}, based on an application of Skorokhod's representation theorem (see \cite[p.70, Theorem 6.7]{Billingsley}) and Skorokhod's results on convergence of stochastic integrals. We refrain here from giving the complete details, which can also be found in \cite[Section 3]{FGL}, and only give the main ideas.

Writing $W=(W^{k,i}_\cdot )_{k,i}=\big\{(W^{k,i}_t)_{0\leq t\leq T} : k\in \mathbb{Z}^d_0, i=1,\ldots, d-1 \big\}$ and denoting by $P_N$ the joint law of the pair $(u_N,W)$, the sequence $\{P_N\}_N$ is tight in
  \[ \mathcal{X}_K\times \mathcal{Y} = \mathcal{X}_K\times C\big([0,T];\mathbb{C}^{\mathbb{Z}^2_0} \big);   \]
therefore we can find another probability space $\big(\tilde{\Omega},\tilde{F},(\tilde{F}_t),\tilde{\mathbb{P}} \big)$ and a sequence $\{(\tilde{u}_N, \tilde{W}_N) \}_N$ such that $(\tilde{u}_N, \tilde{W}_N)\to (\tilde{u},\tilde{W})$ $\tilde{\mathbb{P}}$-a.s. in $\mathcal{X}_K \times \mathcal{Y}$ and $(\tilde{u}_N,\tilde{W}_N)$ is distributed according to $P_N$. It follows that $\tilde{u}_N= u_N(0) + \tilde{v}_N +\tilde{M}_N$ solves the $N$-step SDE associated to $\tilde{W}_N$ and satisfies the same a priori estimates as $u_N$, $M_N$. By the $\tilde{\P}$-a.s. convergence, it can be shown that $\tilde{u}$ is a solution to \eqref{PDE-cut-off} with the Brownian motions $\tilde{W}$; in particular convergence of the nonlinear term follows from $F$ being continuous on $\mathcal{X}_K$ by Lemma \ref{sec3 lemma continuity}.  Estimates \eqref{sec3 energy bound} and \eqref{sec3 estimate martingale} then follow from the analogous ones for $\tilde{u}_N$, $\tilde{M}_N$.

\textbf{Step 3} (pathwise uniqueness and strong existence). Let now $u_1= u_0+ v_1+M_1$, $u_2=u_0+ v_2+M_2$ be two solutions defined on the same probability space, with respect to the same Brownian motions $\{W^k\}_k$ and initial data $u_0\in L^2$; moreover, both $u_1$ and $u_2$ fulfil the bound \eqref{sec3 energy bound}. Then, the difference $\tilde{u}= u_1-u_2= \tilde{v}+ \tilde{M}$ satisfies, for any $\phi \in H^\alpha(\T^2)$,
  \begin{equation*}\begin{split}
  \langle \tilde{u}(t),\phi\rangle
  = & - \int_0^t \langle \Lambda^{\alpha} \tilde{u}(s), \Lambda^{\alpha}\phi \rangle\, \d s - \nu \int_0^t \langle \nabla \tilde{u}(s),\nabla\phi\rangle\, \d s \\
  & + \int_0^t  \langle g_R(u_1(s)) F(u_1(s))-g_R(u_2(s)) F(u_2(s)),\phi \rangle\, \d s\\
  & - \sqrt{C_d\nu} \sum_{k,i} \int_0^t \theta_k \langle \tilde{u}(s),\sigma_{k,i}\cdot \nabla\phi\rangle \,\d W^{k,i}_s.
  \end{split}\end{equation*}
By the same computations as in Lemma \ref{sec3 lemma structure solutions}, the martingale part $\tilde{M}$ has quadratic variation $[\tilde{M}]=2\nu\int_0^\cdot \Vert \nabla \tilde u(s)\Vert_{L^2}^2\, \d s $. We can therefore invoke the It\^o formula in \cite[Theorem 2.13]{RozLot} to deduce that
  \begin{align*}
  \d \Vert \tilde{u}(t)\Vert_{L^2}^2& = - 2\|\Lambda^{\alpha} \tilde u(t)\|_{L^2}^2\, \d t - 2\nu \Vert \nabla\tilde{u}(t)\Vert_{L^2}^2\, \d t \\
  &\quad + 2\big\langle g_R(u_1(t)) F(u_1(t))-g_R(u_2(t)) F(u_2(t)), \tilde{u}(t) \big\rangle\, \d t\\
  & \quad +2\sqrt{C_d\nu} \sum_{k,i} \theta_k \langle \tilde{u}(t),\sigma_{k,i} \cdot \nabla\tilde{u}(t) \rangle\,\d W^{k,i}_t
  + 2\nu \Vert \nabla \tilde u(t)\Vert_{L^2}^2\,\d t,
  \end{align*}
where the last term comes from the quadratic variation of $\tilde{M}$. Using the fact that the vector fields $\sigma_{k,i}$ are divergence free, we arrive at
  \begin{equation} \label{uniqueness.1}
  \d \Vert \tilde{u}(t)\Vert_{L^2}^2 = -2\|\Lambda^{\alpha} \tilde u(t)\|_{L^2}^2\, \d t + 2\big\langle g_R(u_1(t)) F(u_1(t))-g_R(u_2(t)) F(u_2(t)), \tilde{u}(t) \big\rangle \, \d t.
  \end{equation}
The key is to estimate the second term which is dominated by
  $$ 2|g_R(u_1(t)) - g_R(u_2(t))| \, |\<F(u_1(t)), \tilde u(t)\>| + g_R(u_2(t)) |\<F(u_1(t))- F(u_2(t)), \tilde{u}(t) \>| =: I_1 + I_2. $$

First, we estimate $I_1$. The definition of $g_R$ leads to
  $$\aligned
  |g_R(u_1(t)) - g_R(u_2(t))|
  \leq \|g'_R\|_\infty \big|\|u_1(t) \|_{H^{-\delta}} - \|u_2(t) \|_{H^{-\delta}}\|\big|
  \lesssim \|u_1(t)- u_2(t)\|_{H^{-\delta}}
   \lesssim \|\tilde u(t)\|_{L^2}.
  \endaligned $$
Recall that, $\P$-a.s., both solutions $u_i$ satisfy the deterministic bound \eqref{sec3 energy bound};  therefore by (H1),
  $$\aligned
  |\<F(u_1(t)), \tilde u(t)\>|
  &\lesssim \big(1 + \|u_1(t) \|_{L^2}^{\beta_1} \big) (1+\|u_1(t)\|_{H^\alpha}) \|\tilde u(t) \|_{H^\alpha}
  \lesssim (1+\|u_1(t)\|_{H^\alpha}) \|\tilde u(t) \|_{H^\alpha}.
  \endaligned $$
As a result,
\begin{align}
I_1 & \lesssim (1+\|u_1(t)\|_{H^\alpha}) \|\tilde u(t)\|_{L^2} \|\tilde u(t) \|_{H^\alpha} \nonumber \\
& \leq C (1+\|u_1(t)\|_{H^\alpha}^2)\|\tilde u(t)\|_{L^2}^2 + \frac{1}{2^\alpha}\|\tilde u(t) \|_{H^\alpha}^2 \nonumber \\
& \leq C' (1+\|u_1(t)\|_{H^\alpha}^2)\|\tilde u(t)\|_{L^2}^2 + \frac{1}{2}\|\Lambda^{\alpha} \tilde u(t) \|_{L^2}^2  \label{estimate-1}
\end{align}
where we used the Cauchy inequality and the fact that $\Vert \tilde{u}\Vert_{H^\alpha}^2\leq 2^{\alpha-1} \big( \Vert \tilde{u}\Vert_{L^2}^2 +\Vert \Lambda^{\alpha} \tilde{u}\Vert_{L^2}^2 \big)$.

We now turn to deal with $I_2$. The hypothesis (H3) and the uniform bound \eqref{sec3 energy bound} immediately give us
  $$\aligned
  I_2 &\lesssim \|\tilde u(t) \|_{H^\alpha}^{\gamma_3} \|\tilde u(t) \|_{L^2}^{\beta_3} \big(1+ \|u_1(t) \|_{H^\alpha}^\kappa + \|u_2(t) \|_{H^\alpha}^\kappa \big)\\
  &\leq \frac1{2^\alpha} \|\tilde u(t) \|_{H^\alpha}^2 + C \|\tilde u(t) \|_{L^2}^{2\beta_3/(2-\gamma_3)} \Big(1+ \|u_1(t) \|_{H^\alpha}^{2\kappa/(2-\gamma_3)} + \|u_2(t) \|_{H^\alpha}^{2\kappa/(2-\gamma_3)} \Big)
  \endaligned $$
since $\gamma_3<2$. By (H3) we have $2\beta_3/(2-\gamma_3) \geq 2$, $2\kappa/(2-\gamma_3)\leq 2$; using again the bound \eqref{sec3 energy bound}, we obtain
  $$I_2 \leq \frac12\|\Lambda^{\alpha} \tilde u(t) \|_{L^2}^2 + C' \Vert \tilde{u}(t)\Vert_{L^2}^2 \big(1+ \|u_1(t) \|_{H^\alpha}^2 + \|u_2(t) \|_{H^\alpha}^2 \big). $$
Combining this estimate with  \eqref{uniqueness.1} and \eqref{estimate-1} yields
  $$\d \Vert \tilde{u}(t)\Vert_{L^2}^2 \lesssim (1+ \|u_1(t)\|_{H^\alpha}^2 + \|u_2(t) \|_{H^\alpha}^2 ) \|\tilde u(t)\|_{L^2}^2 \,\d t, \quad \Vert \tilde{u}(0)\Vert_{L^2}=0. $$
As the sum in the parentheses on the right-hand side is integrable on $[0,T]$, we conclude that $\Vert \tilde{u}(t)\Vert_{L^2} \equiv 0$, which implies pathwise uniqueness. Strong existence then follows from an application of Yamada--Watanabe theorem, for instance in the version given in \cite{Rock1} for the choice $U=H=L^2$, $V=H^\alpha$, $E=H^{-\alpha}$.
\end{proof}

We are now ready to prove the following intermediate result.

\begin{proposition}\label{claim SPDE}
Consider a sequence $\{u^N_0\}_{N\geq 1}\subset L^2(\mathbb{T}^d)$ converging weakly in $L^2(\mathbb{T}^d)$ to some $u_0$ and let $\{\theta^N\}_{N\geq 1} \subset \ell^2$ be a family of symmetric coefficients satisfying
  \begin{equation*}
  \Vert \theta^N_\cdot\Vert_{\ell^2} = 1\quad \forall\, N,
  \qquad
  \lim_{N\to\infty} \Vert \theta^N_\cdot\Vert_{\ell^\infty}=0;
  \end{equation*}
denote by $u^N$ the unique strong solution to \eqref{PDE-cut-off} associated to $\theta^N$ starting at $u_0^N,\, N\geq 1$. Then,  for every $\varepsilon>0$ and every $p\geq 2$, $u^N$ converges in probability, in the topology of
  $$\mathcal X = L^2(0,T; H^{\alpha-\eps})\cap L^p(0,T; L^2)\cap C([0,T];H^{-\varepsilon}) ,$$
to the unique solution of the deterministic equation
  \begin{equation}\label{sec3 limit equation}
  \begin{cases}
  \partial_{t}u = -\Lambda^{2\alpha} u + \nu \Delta u +g_{R}(u) F(u),\\
  u(0)=u_0.
  \end{cases}
  \end{equation}
\end{proposition}

\begin{proof}
As the sequence $\{u^N_0\}_{N\geq 1}$ is weakly convergent in $L^2(\mathbb{T}^d)$, it is also bounded therein; therefore by bound \eqref{sec3 energy bound} we can find a deterministic constant $C$ such that
\begin{equation}\label{claim SPDE.1}
\sup_{N\in \mathbb{N}} \bigg( \sup_{t\in [0,T]} \Vert u^N(t)\Vert_{L^2}^2 + \int_0^T \Vert \Lambda^{\alpha} u^N(t)\Vert_{L^2}^2\, \d t\bigg)\leq C \quad  \P\mbox{-a.s.}
\end{equation}
Recall that any solution $u^N$ admits a decomposition $u^N= u_0+ v^N+ M^N$, with $v^N$ satisfying
  \begin{equation*}
  \Vert v^N\Vert_{C^{1/2} H^{-\alpha}} \leq \Vert v^N\Vert_{W^{1,2} H^{-\alpha}};
  \end{equation*}
therefore combining \eqref{sec3 bound v} and \eqref{sec3 estimate martingale} with estimate \eqref{claim SPDE.1} above, together with the fact that $\Vert\theta^N\Vert_{\ell^\infty}$ are bounded, we can find values $q>1$, $\beta,\gamma>0$ such that
  \[
  \sup_{N\geq 1} \mathbb{E}\big[ (\Vert u^N\Vert_{L^\infty L^2} + \Vert u^N\Vert_{L^2 H^\alpha} + \Vert u^N\Vert_{C^\gamma H^{-\beta}})^q \big]<\infty.
  \]
Thus, denoting by $\mu^N$ the law of the solution $u^N,\, N\geq 1$, by Lemma \ref{sec3 lemma compactness} the sequence $\{\mu^N\}_N$ is tight in $\mathcal X $; up to choosing $K$ big enough, by the bound \eqref{claim SPDE.1} it is therefore also tight in $\mathcal{X}_K$.  By Prokhorov's Theorem, we can extract a (not relabelled) subsequence such that $\mu^N$ converges weakly to some probability measure $\mu$ in $\mathcal{X}_K$.

In order to conclude, it is enough to show that $\mu=\delta_u$, where $u$ is the unique deterministic solution to \eqref{sec3 limit equation}. Indeed, as the reasoning applies to any weakly convergent subsequence, we deduce that the whole sequence converges in law to $\mu=\delta_u$, from which convergence in probability follows as well.

For any fixed $\phi\in C^\infty (\mathbb{T}^d)$, define a map $T^\phi: \mathcal X_K\to C([0,T];\mathbb{R})$ by
  \[\aligned
  (T^\phi f)(t) &:= \langle f(t),\phi\rangle-\langle u_0,\phi\rangle - \int_0^t g_R (f(s)) \langle F(f(s)),\phi\rangle\, \d s \\
  &\quad\  +\int_0^t \langle f(s),\Lambda^{2\alpha} \phi(s)\rangle\, \d s -\nu \int_0^t \langle f(s),\Delta \phi(s)\rangle\, \d s.
  \endaligned \]
It is immediate to check that, thanks to Lemma \ref{sec3 lemma continuity}, $T^\phi$ is continuous on $\mathcal X_K$; thus by properties of weak convergence it holds $T^\phi_\sharp \mu^N\to T^\phi_\sharp \mu$, where $T^\phi_\sharp \mu=\mu\circ (T^\phi)^{-1}$ denotes the pushforward measure of $\mu$ under $T^\phi$. On the other hand, $T^\phi_\sharp \mu^N$ is the law of $T^\phi u^N$, where by construction $T^\phi u^N=\langle u^N_0-u_0,\phi\rangle + \langle M^N,\phi\rangle$. Recall that $u^N_0$ converges weakly to $u_0$ as $N\to\infty$, and by \eqref{sec3 estimate martingale} it holds
  \[
  \mathbb{E}\bigg[ \sup_{t\in [0,T]} |\langle M^N(t),\phi\rangle|^p \bigg] \lesssim \Vert \phi\Vert_{H^\beta}^p\, \mathbb{E}\Big[ \Vert M^N\Vert_{C^\gamma H^{-\beta}}^p \Big] \lesssim \|\theta^N \|_{\ell^\infty}^p \to 0;
  \]
we deduce that, for any fixed $\phi$, the support of $T^\phi_\sharp \mu$ must be $\{0\}$, the singleton in $C([0,T];\mathbb{R})$. Applying the reasoning to a countable set $\{\phi^n\}_n\subset C^\infty(\mathbb{T}^2)$ dense in $H^\alpha$, we obtain
  \begin{equation*}
  \mu \big(\big\{ f \in \mathcal X_K: T^{\phi_n} f=0 \text{ for all } n\in \mathbb{N} \big\}\big) =0;
  \end{equation*}
but then by a density argument we can conclude that
  \begin{equation*}
  \mu \big( \big\{f\in \mathcal X_K: f \text{ solves } \eqref{sec3 limit equation} \big\}\big) =1
  \end{equation*}
which necessarily implies that $\mu= \delta_u$, where $u$ is the unique deterministic solution to \eqref{sec3 limit equation}. The conclusion then follows.
\end{proof}

Now we can complete the proof of the main result.

\begin{proof}[Proof of Theorem \ref{Thm main}]
Let $\eps, T>0$ be fixed; by hypothesis (H4), we can find $\nu,R>0$ such that
\begin{equation}\label{main thm proof eq1}
\sup_{u_0\in \mathcal K} \sup_{t\in [0,T]} \Vert u(t;u_0,\nu)\Vert_{L^2}\leq R-1;
\end{equation}
from now on such parameters $\nu$ and $R$ will be fixed. For any $\theta\in \ell^2$ with $\Vert \theta\Vert_{\ell^2}=1$, denote by $u^R(\,\cdot\,; u_0,\nu,\theta)$ the unique global stochastic solution to \eqref{PDE-cut-off} with initial data $u_0$; similarly, $u^R(\,\cdot\,; u_0,\nu)$ denotes the unique global deterministic solution to \eqref{sec3 limit equation}. By estimate \eqref{main thm proof eq1} it follows that
  \begin{equation} \label{main thm proof eq1.5}
  u^R(\,\cdot\,; u_0,\nu)=u(\,\cdot\,; u_0,\nu),\quad u^R(\,\cdot\,; u_0^l,\nu)=u(\,\cdot\,; u_0^l,\nu)\quad \forall\,l\in\mathbb{N}.
  \end{equation}

Next, consider a sequence $\{\theta^N\}\subset\ell^2$ satisfying the assumptions of Proposition \ref{claim SPDE}; we claim that
\begin{equation}\label{main thm proof eq2}
\lim_{N\to\infty} \sup_{u_0\in \mathcal K} \mathbb{P}\Big(\sup_{t\in [0,T]} \Vert u^R(t;u_0,\theta^N,\nu)-u(t;u_0,\nu) \Vert_{H^{-\delta}} >1\Big)=0.
\end{equation}
Indeed, suppose this is not the case; then we can find $\gamma>0$, a sequence $\{N_l\}_l\subset \mathbb{N}$ and a sequence $\{u_0^l\}_l\subset \mathcal K$ such that
\begin{equation}\label{main thm proof eq3}
\mathbb{P}\Big(\sup_{t\in [0,T]} \big\Vert u^R\big(t;u^l_0,\theta^{N_l},\nu \big)-u(t;u^l_0,\nu) \big\Vert_{H^{-\delta}} >1\Big)>\gamma>0 \quad \forall\, l\in\mathbb{N}.
\end{equation}
Since $\mathcal K$ is convex, closed and bounded, it is weakly compact in $L^2(\mathbb{T}^d)$, thus we can find a subsequence (not relabelled for simplicity) such that $u_0^l$ converge weakly to $u_0\in \mathcal K$. It follows from Proposition \ref{claim SPDE} that $u^R(\,\cdot\,;u^l_0,\theta^N,\nu)$ converge in probability in $C([0,T];H^{-\delta})$ to $u^R(\,\cdot\,;u_0,\nu)$ as $l\to\infty$; similarly, it is easy to check that $u^R(\,\cdot\,;u_0^l,\nu)$ converge to $u^R(\,\cdot\,;u_0;\nu)$ in $C([0,T];H^{-\delta})$. We deduce from \eqref{main thm proof eq1.5} that $u^R(\,\cdot\,;u^l_0,\theta^N,\nu)-u(\,\cdot\,;u^l_0,\nu)$ converges to $0$ in probability in $C([0,T];H^{-\delta})$, which is in contradiction with \eqref{main thm proof eq3}. Thus \eqref{main thm proof eq2} holds.

As a consequence of \eqref{main thm proof eq1} and \eqref{main thm proof eq2} we deduce that
\begin{align*}
&\, \lim_{N\to\infty}\sup_{u_0\in \mathcal K} \mathbb{P} \Big(\sup_{t\in [0,T]} \Vert u^R(t;u_0,\theta^N,\nu)\Vert_{H^{-\delta}} >R\Big)\\
\leq &\, \lim_{N\to\infty} \sup_{u_0\in \mathcal K} \mathbb{P}\Big(\sup_{t\in [0,T]} \Vert u^R(t;u_0,\theta^N,\nu)-u(t;u_0,\nu) \Vert_{H^{-\delta}} >1\Big)=0,
\end{align*}
and hence we can find $N$ big enough such that, uniformly over $u_0\in \mathcal K$, it holds
\begin{equation*}
\mathbb{P} \Big( g_R\big( u^R\big(t; u_0,\theta^N,\nu\big)\big)=1 \text{ for all } t\in [0,T] \Big)
= \mathbb{P} \Big(\sup_{t\in [0,T]} \Vert u^R(t;u_0,\theta^N,\nu)\Vert_{H^{-\delta}} \leq R\Big)
>1-\varepsilon.
\end{equation*}
The condition ``$g_{R}(u^R(t;u_0,\theta^N,\nu))  =1$ for all
$t\in\left[  0,T\right]  $'' means that $u^R(\,\cdot\,;u_0,\theta^N,\nu)$ solves, on
$\left[  0,T\right]  $, the stochastic equation without cut-off. This proves
Theorem \ref{Thm main}.
\end{proof}

\section{Appendix - A triviality result}

It is shown in Section \ref{sec3} that given any $\theta\in\ell^2$, the associated stochastic equation with cut-off \eqref{PDE-cut-off} is well posed, with a priori estimates \eqref{sec3 energy bound}, \eqref{sec3 estimate martingale} which do not depend on $\Vert \theta\Vert_{\ell^2}$ but only on $\Vert \theta\Vert_{\ell^\infty}$. Therefore, in principle, instead of taking the limit as $\theta^N\to 0$ in $\ell^\infty$ with $\theta^N$ bounded in $\ell^2$, we could investigate the opposite regime in which $\theta^N$ stay bounded in $\ell^\infty$ with $\Vert \theta^N\Vert_{\ell^2}\to \infty$. We show that in this case any limit is necessarily trivial, i.e. spatially constant, similarly to what has been established in a different setting in \cite[Theorem 1.3]{FL19b}.

For simplicity we only treat the case of an $\mathbb{R}$-valued solution of \eqref{intro SPDE} with $\alpha=1$, but the result easily generalizes to the case of vector-valued solutions and to different values $\alpha\geq 1$.

\begin{theorem}\label{appendix main theorem}
Let $\{\theta^N\}_N\subset \ell^2(\mathbb{Z}^d_0)$ be a sequence of symmetric coefficients such that
\begin{equation}\label{appendix assumption coefficients}
\sup_N \Vert \theta^N\Vert_{\ell^\infty}<\infty, \quad
\lim_{N\to\infty} \Vert \theta^N\Vert_{\ell^2}=+\infty.
\end{equation}
Fix $u_0\in L^2(\mathbb{T}^d)$ and denote by $u^N$ the unique global solution in $C(0,T;L^2)\cap L^2(0,T;H^1)$ to
\begin{equation}\label{appendix SPDE}
\begin{cases}
\d u^N = \big[ \Delta u^N + g_{R}(u^N) F(u^N)\big]\,\d t
+ \sqrt{C_d\nu} \sum\limits_{k,i} \theta^N_{k}\, \sigma_{k,i}\cdot\nabla u^N \circ \d W_t^{k,i},\\
u^N(0)= u_{0},
\end{cases}
\end{equation}
where $\nu,R>0$ are fixed parameters. Then $\{u^N\}_N$ is a.s. bounded in $C(0,T;L^2)\cap L^2(0,T;H^1)$, and any weakly-$\ast$ convergent subsequence of $\{u^N\}_N$ converges to some trivial limit $u$ which is a spatially constant process on $(0,T)$.
\end{theorem}

Unfortunately, we are unable to show that the constant is independent of $t\in (0,T)$, unless we assume that $\int_{\T^d} F(u)\,\d x =0$ for any $u\in H^1$.

\begin{proof}[Proof of Theorem \ref{appendix main theorem}]
By \eqref{sec3 energy bound} we can find a deterministic constant $K$ such that
\begin{equation}\label{proof-appendix-1}
\sup_N \Big( \sup_{t\in [0,T]} \Vert u^N(t)\Vert_{L^2} + \Vert u^N\Vert_{L^2 H^1} \Big)\leq K\quad \mathbb{P}\text{-a.s.}
\end{equation}
Let $\{u^{N_i}\}_{i\geq 1}$ be any weakly-$\ast$ convergent subsequence in $L^\infty\big(\Omega, L^\infty(0,T; L^2(\T^d))\big)$ with a limit $u$; then, for any $\xi\in L^\infty(\Omega, \mathcal F, \P)$, $f\in C([0,T])$ and $\phi\in L^2(\T^d)$, one has
  \begin{equation}\label{proof-appendix-1.5}
  \lim_{i\to\infty} \E\bigg[\xi \int_0^T f(t) \<u^{N_i}(t), \phi\>\,\d t \bigg] = \E\bigg[\xi \int_0^T f(t) \<u(t), \phi\>\,\d t \bigg].
  \end{equation}

Rewriting equation \eqref{appendix SPDE} in It\^o integral weak form, $u^{N_i}$ satisfies, for any $\phi\in C^2(\T^d)$,
\begin{align*}
\langle u^{N_i}(t),\phi\rangle - \langle u_0,\phi\rangle
& = \big(1+\nu \Vert \theta^{N_i}\Vert_{\ell^2}^2 \big) \int_0^t  \langle u^{N_i}(s), \Delta \phi \rangle\, \d s \\
& \quad + \int_0^t g_R(u^{N_i}(s)) \langle F(u^{N_i}(s)), \phi\rangle\, \d s + \langle M^{N_i} (t),\phi\rangle.
\end{align*}
Setting $\lambda^{N_i}:= 4\pi^2 \big(1+\nu \Vert \theta^{N_i}\Vert_{\ell^2}^2 \big)$, choosing $\phi(x)=e_k(x)=e^{2\pi {\rm i} k\cdot x}$ with $k\in \Z^d_0$ and rearranging the terms, we get
\begin{align*}
\int_0^t \langle u^{N_i}(s),e_k\rangle \,\d s
& = \frac1{\lambda^{N_i} |k|^2} \bigg[ \langle u_0- u^{N_i}(t),e_k\rangle + \int_0^t g_R(u^{N_i}(s)) \langle F(u^{N_i}(s)), e_k \rangle\, \d s \bigg]\\
&\quad +\frac1{\lambda^{N_i} |k|^2} \langle M^{N_i}(t), e_k\rangle,
\end{align*}
so that
\begin{equation}\label{proof-appendix-2}
\aligned
&\quad\ \mathbb{E}\bigg[\sup_{t\in [0,T]} \Big| \int_0^t \langle u^{N_i}(s),e_k\rangle\, \d s\Big| \bigg] \\
& \leq \frac1{\lambda^{N_i} |k|^2} \mathbb{E}\bigg[ \sup_{t\in [0,T]} \big|\langle u_0- u^{N_i}(t) , e_k\rangle \big| + \int_0^T g_R(u^{N_i}(s)) \big| \langle F(u^{N_i}(s)), e_k \rangle \big|\, \d s \bigg] \\
&\quad + \frac1{\lambda^{N_i} |k|^2} \E\bigg[\sup_{t\in [0,T]} \big|\langle M^{N_i}(t), e_k\rangle \big| \bigg].
\endaligned
\end{equation}
By \eqref{proof-appendix-1}, we have $|\langle u^{N_i}(t),e_k\rangle| \leq \|u^{N_i}(t) \|_{L^2} \leq K$ $\P$-a.s. for all $t\in [0,T]$; and by hypothesis (H1), it holds that
\begin{align*}
\int_0^T g_R(u^{N_i}(s)) \big| \langle F(u^{N_i}(s)), e_k \rangle \big|\, \d s
& \lesssim |k| \int_0^T \Vert F(u^{N_i}(s))\Vert_{H^{-1}} \d s\\
&\lesssim |k| \int_0^T \big(1+\Vert u^{N_i}(s)\Vert_{H^1} \big) \big(1+\Vert u^{N_i}(s)\Vert_{L^2}^{\beta_1} \big)\, \d s \\
&\lesssim |k| (1+K^{1+\beta_1}).
\end{align*}
Thus, the first term on the right-hand side of \eqref{proof-appendix-2} is dominated by
\begin{equation*}
\frac1{\lambda^{N_i} |k|^2} \big[2K + |k| (1+K^{1+\beta_1})\big] \lesssim \frac{1+K^{1+\beta_1}}{\lambda^{N_i} |k|} \to 0 \quad \text{ as } i \to \infty
\end{equation*}
since $\lambda^{N_i}\to \infty$ by assumption \eqref{appendix assumption coefficients}.

Next, notice that
  $$\langle M^{N_i}(t), e_k\rangle = -\sqrt{C_d \nu} \sum_{l,j} \theta^{N_i}_l \int_0^t \<u^{N_i}(s), \sigma_{l,j}\cdot \nabla e_k\> \,\d W^{l,j}_s, $$
where $l$ runs over $\Z^d_0$ and $j$ over $\{1,\ldots, d-1\}$. We have by Burkholder's inequality,
  $$\E\bigg[\sup_{t\in [0,T]} \big|\langle M^{N_i}(t), e_k\rangle \big| \bigg] \leq C_{d,\nu}\, \E \bigg[ \Big(\int_0^T \sum_{l,j} \big(\theta^{N_i}_l \big)^2 \big|\<u^{N_i}(s), \sigma_{l,j}\cdot \nabla e_k\>\big|^2 \,\d s \Big)^{1/2} \bigg]. $$
Using the fact that $\{\sigma_{l,j} \}_{l,j}$ is an (incomplete) orthonormal system in $L^2(\T^d)$, we obtain
  $$\aligned
  \sum_{l,j} \big(\theta^{N_i}_l \big)^2 \big|\<u^{N_i}(s), \sigma_{l,j}\cdot \nabla e_k\>\big|^2 &\leq \|\theta^{N_i} \|_{\ell^\infty}^2 \sum_{l,j} \big|\<u^{N_i}(s)\nabla e_k, \sigma_{l,j} \>\big|^2 \\
  & \leq \|\theta^{N_i} \|_{\ell^\infty}^2  \|u^{N_i}(s)\nabla e_k\|_{L^2}^2 \lesssim \|\theta^{N_i} \|_{\ell^\infty}^2 |k|^2 \|u^{N_i}(s) \|_{L^2}^2.
  \endaligned $$
Therefore,
  $$\E\bigg[\sup_{t\in [0,T]} \big|\langle M^{N_i}(t), e_k\rangle \big| \bigg] \leq C_{d,\nu} \|\theta^{N_i} \|_{\ell^\infty} |k|\, \E\big( \|u^{N_i} \|_{L^2 L^2} \big) \leq C_{d,\nu} \|\theta^{N_i} \|_{\ell^\infty} |k| K,  $$
where the last step follows from \eqref{proof-appendix-1}. As a result, using the definition of $\lambda^{N_i}$,
  $$\frac1{\lambda^{N_i} |k|^2} \E\bigg[\sup_{t\in [0,T]} \big|\langle M^{N_i}(t), e_k\rangle \big| \bigg] \leq \frac{C_{d,\nu} \|\theta^{N_i} \|_{\ell^\infty} K}{4\pi^2 \big(1+\nu \Vert \theta^{N_i}\Vert_{\ell^2}^2 \big) |k|}, $$
which, by \eqref{appendix assumption coefficients}, vanishes as $i\to \infty$. To sum up, we have shown that both terms on the right-hand side of \eqref{proof-appendix-2} tend to 0 as $i\to \infty$, and hence
  \begin{equation}\label{proof-appendix-3}
  \lim_{i\to \infty} \mathbb{E}\bigg[\sup_{t\in [0,T]} \Big| \int_0^t \langle u^{N_i}(s),e_k\rangle\, \d s\Big| \bigg] =0 \qquad \forall\,k\in\mathbb{Z}^d_0.
  \end{equation}

Now, for any $f\in C_c^1((0,T))$, integrating by parts  yields
  $$\aligned \int_0^T f(t) \<u^{N_i}(t), e_k\>\,\d t &= \Big( f(t) \int_0^t\<u^{N_i}(s), e_k\>\,\d s\Big)\Big|_{t=0}^T - \int_0^T f'(t) \Big( \int_0^t\<u^{N_i}(s), e_k\>\,\d s\Big) \,\d t \\
  &=  - \int_0^T f'(t) \Big( \int_0^t\<u^{N_i}(s), e_k\>\,\d s\Big) \,\d t.
  \endaligned $$
Next, since
  $$\bigg| \int_0^T f'(t) \Big( \int_0^t\<u^{N_i}(s), e_k\>\,\d s\Big) \,\d t\bigg| \leq \|f'\|_{L^1((0,T))} \sup_{t\in [0,T]} \Big| \int_0^t\<u^{N_i}(s), e_k\>\,\d s\Big|, $$
we have, for any $\xi\in L^\infty(\Omega, \mathcal F, \P)$,
  $$\aligned
  \bigg|\E\bigg[\xi \int_0^T f(t) \<u^{N_i}(t), e_k\>\,\d t \bigg] \bigg| &\leq \|\xi \|_{L^\infty(\Omega)}\, \E \Big|\int_0^T f(t) \<u^{N_i}(t), e_k \>\,\d t \Big| \\
  &\leq \|\xi \|_{L^\infty(\Omega)} \|f'\|_{L^1((0,T))}\, \E\bigg[ \sup_{t\in [0,T]} \Big| \int_0^t\<u^{N_i}(s), e_k\>\,\d s\Big| \bigg]
  \endaligned $$
which, by \eqref{proof-appendix-3}, vanishes as $i\to \infty$. Combining this with \eqref{proof-appendix-1.5} gives us
  $$\E\bigg[\xi \int_0^T f(t) \<u(t), e_k\>\,\d t \bigg]=0. $$
The arbitrariness of $\xi\in L^\infty(\Omega, \mathcal F, \P)$ implies that, $\P$-a.s.,
  $$\int_0^T f(t) \<u(t), e_k\>\,\d t =0 $$
for any $f\in C_c^1((0,T))$ and $k\in \Z^d_0$. Since $\Z^d_0$ is countable, and taking a countable set $\{f_n\}_n$ which is dense in $ C_c^1((0,T))$,  we conclude that, $\P$-a.s., for a.e. $t\in (0,T)$, $\<u(t), e_k\> =0$ for all $k\in \Z^d_0$. This shows the triviality of the limit $u$.
\end{proof}

\bigskip

\noindent \textbf{Acknowledgements.} The authors are very grateful to the referee for reading carefully the paper and for many valuable comments. The last named author would like to thank the financial supports of the National Key R\&D Program of China (No. 2020YFA0712700) and the National Natural Science Foundation of China (Nos. 11688101, 11931004, 12090014).

\end{document}